\documentclass[10pt,reqno]{article}
\usepackage{amssymb}
\usepackage{enumitem}
\usepackage{amsmath,amsbsy,amsfonts,amsthm}
\usepackage{graphicx}
\usepackage{geometry}
\usepackage{cancel}
\usepackage{hyperref}
\usepackage{cleveref}
\usepackage{color}
\usepackage[utf8]{inputenc}

\newtheorem{theorem}{Theorem}[section]
\newtheorem{lemma}[theorem]{Lemma}
\newtheorem{proposition}[theorem]{Proposition}
\newtheorem{corollary}[theorem]{Corollary}
\newtheorem{remark}[theorem]{Remark}
\newtheorem{definition}[theorem]{Definition}

\begin{document}

\numberwithin{equation}{section}

\title{Regularizing Effect for a Class of Maxwell-Schr\"odinger Systems}
\author{Ayana Pinheiro de Castro Santana and Lu\'is Henrique de Miranda}
%\author{Author 1 \footnote{ Universidade de Brasília,
		% Universidade de Bras\'{\i}lia,
		% Department of Mathematics,
		% 70910-900, Bras\'{\i}lia-DF, Brazil, ayana.santana@gmail.com; (site) Partially supported by CAPES/Brazil-Finance Code 001.
		% } \qquad Keti Miranda\footnote{ Universidade de Bras\'{\i}lia,
		% Department of Mathematics,
		% 70910-900, Bras\'{\i}lia-DF, Brazil, (email) Partially supported by CNPq Proc. 312462/2014-0, Ministry of Science and Technology, Brazil and CAPES/Brazil-Finance Code 001.}
	
	%  and FAPDF/Brazil grant 0193.001346/2016.
	
	% }

\date{}

\maketitle{}

\begin{abstract}
	\hspace{-.5cm} In this paper we prove existence and regularity of weak solutions for the following system
	\begin{align*}
		\begin{cases}
			-\mbox{div}(M(x)\nabla u) + g(x,u,v) = f \ \ \mbox{in} \ \ \Omega\\
			-\mbox{div}(M(x)\nabla v) = h(x,u,v)  \ \ \mbox{in} \ \ \Omega\\
			\ \ \ \ \ u=v=0 \ \ \mbox{on} \ \ \partial \Omega,
		\end{cases}
	\end{align*}
	where $\Omega$ is an open bounded subset of $\mathbb{R}^N$, for $N>2$,  $f\in L^m(\Omega)$, $m>1$ and $g$, $h$ are two Carath\'eodory functions. We prove that under appropriate conditions on $g$ and $h$, there exist solutions which escape the predicted regularity by the classical Stampacchia's theory causing the so-called regularizing effect.

%\begin{abstract}
%A rigidity result for a class of compact generalized quasi-Einstein manifolds  with constant scalar curvature is obtained. Moreover, under some geometric assumptions, the rigidity for the noncompact case is also proved. Considering non constant scalar curvature, we characterize the generalized quasi-Einstein manifolds which is conformal to the Euclidean space and we show that there exist two classes of complete manifolds, which  are obtained by considering potential functions and conformal factors either to be radial  or invariant under the action of an (n-1) dimensional translation group. Explicit examples are given.
%\\  \\
%\noindent
%{\em Keywords: }
%Generalized quasi-Einstein Manifolds, Einstein type manifolds, Rigidity results, Conformally flat complete manifolds.   \\
%\noindent
\end{abstract}

\maketitle

%for our system and address the so-called regularizing effect due to the coupling between the equations in \eqref{P}
%In the present paper, we are interested in the investigation of the regularity properties for the positive solutions
\section{Introduction}\label{Intro}
In the present work, we investigate existence and regularity of positive solutions to the system
\begin{align*}
\tag{P}\label{P}
		\begin{cases}
			-\mbox{div}(M(x)\nabla u) + g(x,u,v) = f \ \ \mbox{in} \ \ \Omega\\
			-\mbox{div}(M(x)\nabla v) = h(x,u,v)  \ \ \mbox{in} \ \ \Omega\\
			\ \ \ \ \ u=v=0 \ \ \mbox{on} \ \ \partial \Omega,
		\end{cases}
	\end{align*}
%Where $\Omega \subset\mathbb{R}^{N}$ is an open  bounded subset, with $N\geqslant3$, $f\in L^m(\Omega)$ with $m\geqslant1$ and $r>1$.
which under adequate assumptions, encompasses the so--called Maxwell-Schr\"odinger system. The general idea regarding these systems is that due to the strong coupling between both equations, solutions have zones where they are more regular than expected from the classical regularity theory, the so--called Regularizing Effect. %The existence of this phenomenon for Maxwell-Schr\"odinger systems was discovered  in \cite{Boccardo1} by L. Boccardo and since then has been addressed by D. Arcoya, L. Orsina, and R. Durastani, among others, where we refer the reader to  \cite{ABO, Boccardo0, Boccardo2, Boccardo3, Boccardo4, Durastanti} and the references therein.
 This phenomenon has been studied after \cite{Benci} by V. Benci and Fortunato, where the problem was first addressed, and later on in since the seminal work \cite{Boccardo1} followed by several interesting contributions \cite{ABO, Boccardo0, Boccardo2, Boccardo3, Boccardo4, Durastanti} produced by Boccardo, Orsina, Arcoya, Durastanti, among others.
Following \cite{Boccardo1}, %The basic idea is that the solutions to a certain class of problems are more regular than what would be guaranteed by the standard regularity results for the decoupled equations of its system.
for instance, it happens that even when the data $f$ is very irregular, e.g.,  $f \notin (L^{2^*}(\Omega))^\prime$, it is possible to guarantee the existence of energy solutions in $W^{1,2}_0(\Omega)$, see also \cite{Boccardo2} and the references therein.

In other to clarify the ideas and to present some of the background concerning Regularizing Effects for Maxwell--Schr\"odinger equations, let us briefly discuss the papers which have most inspired the present work, namely \cite{Boccardo1, Boccardo2, Boccardo} and \cite{Durastanti}. We start by considering the following  Dirichlet problem
\begin{align*}\tag{1.1}\label{P0}
\begin{cases}
&-\mbox{div}(M(x)\nabla u) + A v |u|^{r-2}u=f, \ \ \  \ u \ \ \mbox{in} \ W^{1,2}_0(\Omega); \\
&-\mbox{div}(M(x)\nabla v) = |u|^r,\  \ \ \ \ \  \ \ \ \ \  \ \ \  \ v  \ \ \mbox{in} \ W^{1,2}_0(\Omega).
\end{cases}
\end{align*}\
for $f \in L^m(\Omega)$, with $m \geqslant 1$, $A>0$, $M(x)$ is a  uniformly elliptic bounded measurable symmetric matrix, and $\Omega\subset \mathbb{R}^N$, $N>2$, is an open bounded domain.  %such that $M(x)\xi\cdot \xi \geqslant \alpha|\xi|^2$, $|M(x)|\leqslant \beta$ for almost every $x$ in $\Omega$ and every $\xi$ in $\mathbb{R}^N$, with $0<\alpha\leqslant\beta$.
In \cite{Boccardo1}, for a more general class of equations, the author proved that if $f\in L^m(\Omega)$, with $m \geqslant (2^*)^{'}$, then there exists an energy solution $(u,v) \in W^{1,2}_0(\Omega)\times W^{1,2}_0 (\Omega)$, despite that the right-hand side of the second equation does not belong to the dual space $W^{-1,2}(\Omega)$. For this, the strategy is first to obtain a priori estimates by a clever choice of test functions to an approximated  variational PDE, then the proof is finished by means of standard compactness arguments. Moreover, for the specific case of \eqref{P0}, the author found out additional regularizing  zones for the parameters. As a matter of fact, it was shown that for $2 \leqslant m < \frac{(r-1)N}{2r}$, with $r>\frac{2^{*}}{2}$,  then $u\in L^{rm}(\Omega)$. Remark that, in the light of the classical Stampacchia's theory, see \cite{Stam1,Stam2}, by regarding $u$ solely as the solution of the first equation of \eqref{P0}, its expected regularity would be $L^{m^{**}}(\Omega)$. However, under the latter conditions $rm> m^{**}$, so that the coupling in \eqref{P0} gives $u$ some extra regularity. Later on, in \cite{Boccardo2}, the authors refined the latter result. Indeed, they proved that:
\begin{itemize}
    \item[(i)] if $r>2^{*}$ and $f \in L^m(\Omega)$, with $r'\leqslant m < (2^{*})'$, then \eqref{P0} has a solution $(u,v) \in W^{1,2}_0(\Omega)\times W^{1,2}_0(\Omega)$ where $u\in L^{\tau_1}(\Omega)$ and $\tau_1=\max(m(r-1),m^{**})$;
    \item[(ii)] if $1<r<2^*$ and $\max\bigg( \frac{Nr}{N+2r},1\bigg)<m < (2^{*})'$, then \eqref{P0} has a solution $(u,v) \in W^{1,m^*}_0(\Omega)\times W^{1,\tau_2}_0(\Omega)$ where $\tau_2=\min\big(\frac{Nm}{Nr-2mr-m},2\big)$.
\end{itemize}
Succeeding, \cite{Boccardo1,Boccardo2}, in  \cite{Durastanti}, the author proposes the study of the following nonlinear elliptic system
\begin{align}\tag{1.2}\label{PP0}
\begin{cases}
&-\mbox{div}(|\nabla u|^{p-2}\nabla u) + A v^{\theta+1} |u|^{r-2}u=f, \ \ \  \ u \geqslant 0 \ \ \mbox{in} \ W^{1,p}_0(\Omega); \\
&-\mbox{div}(|\nabla v|^{p-2}\nabla v) = |u|^rv^\theta,\  \ \ \ \ \  \ \ \ \ \  \ \ \  \ v \geqslant 0 \ \ \mbox{in} \ W^{1,p}_0(\Omega),
\end{cases}
\end{align}
where $1< p <N$ and $0\leqslant \theta< p-1$. Remark that although it is a $p$-Laplacian system, for $\theta =0 $ its zeroth order nonlinear term reduce to \eqref{P0}. We stress that, for the case $\theta = 0$, the author shows existence and Regularizing Effects even the source $f$ not belonging to $W^{-1,p^\prime}(\Omega)$. Actually, the author proves that if
 $f \in L^m(\Omega)$ with $ (r+1)^\prime \leqslant m < (p^*)^\prime$ there exists a weak
solution $(u, v) \in W^{1,p}_0(\Omega) \times W^{1,p}_0(\Omega)$ of system \eqref{PP0}. Further, the author also shows  existence in the dual case, for $f \in L^m(\Omega)$ with $m \geqslant (p^*)^\prime$, there exists $(u, v) \in W^{1,p}_0(\Omega) \times W^{1,p}_0(\Omega)$ solution, with $A>0$ , $r>1$ and $0\leqslant \theta < p -1$. Let us remark that if $\theta \in (0,1)$,  the second equation in \eqref{PP0} is sublinear, raising additional difficulties to the problem, see \cite{brezis, BocOr1994, DuraOliv2022} and the references therein.

%Nevertheless, the method used to prove the results established in \cite{Durastanti} were not sufficient to guarantee a regularizing effect for the regularity of the solution of the fist equation  in the case where $\theta>0$.\\

In addition to the latter results, there have been other contributions to the investigation of regularizing effects phenomena in general. For instance, without the ambition of being complete, we refer the reader to \cite{Boccardo3}, which is divided in two parts, where the first one consists in a survey for the theory and in the second one some contributions for different classes of Dirichlet systems are presented. Further, we mention  \cite{Boccardo4}, for regularizing effects of positive solutions of a symmetric version of \eqref{P0}, or \cite{ABO} where under an interesting growth assumption for the source term the authors obtain extra regularity for the solutions.

Regarding  the present paper, we have decided to address \eqref{P} and to revisit part of the questions raised in the past theory adapted for our problem.
If on one hand, sometimes our results are valid for a different kind of differential operators, e.g. if we contrast to \cite{Durastanti}, on the other hand, our contributions are valid for a different class of zero order nonlinearities, e.g., see \cite{Boccardo4}. Yet, despite that we deal with a slightly different type of system, we tried to investigate  certain aspects of the theory which were not fully disclosed, at least for \eqref{P}, to the best of our knowledge. More precisely, in order to better explain our main results, late us state our basic hypotheses.
\subsection{Assumptions}

%\begin{align*}
%\begin{cases}
%&-\mbox{div}(M(x)\nabla u)  + h(x,u,v)=f \  \ \ %\mbox{in} \ \Omega; \\
%&-\mbox{div}(M(x)\nabla v)  = g(x,u,v) \  \ \ \ %\mbox{in} \ \Omega; \\
%&u=v=0 \ \mbox{on} \ \partial\Omega.
%\end{cases}
%\end{align*}\
%For us, $\Omega \subset\mathbb{R}^{N}$ is an open  bounded subset, with $N\geqslant3$, $f\in L^m(\Omega)$ with $m\geqslant1$, $r>1$.

Below, we describe the basic assumptions for our manuscript. Indeed, throughout the entire text we will always assume that
$\Omega \subset\mathbb{R}^{N}$ is an open  bounded subset, where $N>2$. Remark that ask no smoothness on $\partial \Omega$. We also consider the real parameters $r>1$ and $0<\theta <\min \big\{\frac{4}{N-2},1\big\}$, and we take $f\in L^m(\Omega)$, for $m> 1$.

Regarding the semilinear part of System  \eqref{P}, we consider that  $g, h:\Omega\times\mathbb{R}\times\mathbb{R}\to \mathbb{R}$ are both Carath\'{e}odory  satisfying the properties below:
\begin{enumerate}
\item[(a)]\label{P_1} there exist $c_1$, $c_2>0$ such that
\begin{align}
\label{P1}\tag{$\mbox{P}_1$}
c_1|s|^{r-1}|t|^{\theta+1} \leqslant |g(x,s,t)| \leqslant c_2|s|^{r-1}|t|^{\theta+1};
\end{align}
\item[(b)] $g(x,s,t)$ is monotone in $s$, i.e., \begin{align}\label{P2}\tag{$\mbox{P}_2$}
(g(x,s_1,t)-g(x,s_2,t))(s_1-s_2)\geqslant0 \ \ \forall \ s_1,s_2 \in\mathbb{R}, t \geqslant 0  \ \ \mbox{a.e.}\  x\ \mbox{in} \ \Omega;
\end{align}
\item[(c)] there exist $d_1$, $d_2>0$ such that
\begin{align}
\label{P3}\tag{$\mbox{P}_3$}
d_1|s|^{r}|t|^{\theta} \leqslant |h(x,s,t)| \leqslant d_2|s|^{r}|t|^{\theta};
\end{align}
\item[(d)] $h(.,.,.)$ is non-negative
\begin{align}\label{P4}\tag{$\mbox{P}_4$}
h(x,s,t)\geqslant0 \ \ \forall \ s,t \in\mathbb{R}, \ \ \mbox{a.e.}\  x\ \mbox{in} \ \Omega.
\end{align}
%\item[(e)] $M(x)$ is a symmetric measurable matrix such that there exists $ \alpha ,\beta \in \mathbb{R}^+$ satisfying
\end{enumerate}
Remark that by \eqref{P1}, $g(x,0,t)=0$ for all $t\in \mathbb{R}$ and a.e. $x\in \Omega$, so that by \eqref{P2}, there holds that
\begin{equation}
\label{P22}\tag{$\mbox{P}_2^\prime$}
g(x,s,t)s\geqslant0 \ \ \forall \ s \in\mathbb{R}, t \geqslant 0 \ \ \mbox{a.e.}\  x\ \mbox{in} \ \Omega;
\end{equation}
Regarding the differential operators in \eqref{P}, we assume that $M(x)$ is a symmetric measurable matrix such that $M \in W^{1,\infty}(\Omega)$ and there exist $ \alpha ,\beta \in \mathbb{R}^+$ satisfying
\begin{align}\label{P5}\tag{$\mbox{P}_5$}
\alpha|\xi|^2 \leqslant M(x)\xi \cdot \xi \ , \ |M(x)|\leqslant \beta\ \
\mbox{for every}  \ \ \xi\in \mathbb{R}^N.
\end{align}

At this point, we must emphasize that as a natural example of nonlinearities satisfying \eqref{P1}-\eqref{P4} we may consider $g(x,s,t)=|s|^{r-2}s t$ and $h(x,s,t)=|s|^r$, the classical Maxwell-Schr\"odinger case, or $g(x,s,t)=|s|^{r-2} |t|^{\theta}st$ and $h(x,s,t)=|s|^r|t|^\theta$. For another interesting example, consider $\eta>0$, $V_i(x)$,  bounded and measurable, such that $V_i(x)\geqslant e_i>0$ a.e. in $\Omega$, $i=1,2$, and then set
\[ g(x,s,t)= \dfrac{V_1(x) |s|^{r-2}  |t|^\theta (s^2(\eta+1)+\eta t^2)s t}{s^2+t^2} \ \mbox{ and } h(x,s,t)= \dfrac{V_2(x) |s|^r|t|^{\theta}(s^2+t^2(\eta+1))}{s^2+t^2}.\]

Concerning problem \eqref{P}, in our manuscript, we will consider the following definition of solution:
\begin{definition}\label{defi1}
We say that $(u,v)$ in $W^{1,t}_{0}(\Omega)\times W^{1,t}_{0}(\Omega)$,  for $t\geqslant 1$, is a distributional solution for problem \eqref{P} if and only if
\begin{align*}\tag{$P_{F}$}\label{P_F}
\begin{cases}
&\int_{\Omega}M(x)\nabla u\cdot\nabla\varphi  + \int_{\Omega} g(x,u,v)\varphi = \int_{\Omega} f\varphi \ \ \ \ \forall\ \varphi \in C^\infty_c(\Omega)\\
& \int_{\Omega}M(x)\nabla v\cdot\nabla\psi = \int_{\Omega} h(x,u,v)\psi \ \ \ \ \ \forall\ \psi \in C^\infty_c(\Omega).
\end{cases}
\end{align*}
where $g(.,u,v)\in L^1_{loc}(\Omega)$ and $h(.,u,v)\in L^1_{loc}(\Omega)$.
\end{definition}

From now on, $C>0$ will denote a general constant, which may vary from line to line, and may depend only on the data, i.e.,
$C=C\big(\alpha,\beta,\theta, \Omega, c_1,c_2,d_1,d_2,r ,N\big)>0$. Sometimes, in order to simplify the notation, we will denote $C=C(f)>0$ in order to stress that $C$ depends on $\|f\|_{L^m}$.

For us, given $k>0$, $T_k$ and $G_k$ stand for the standard Stampacchia's truncations, i.e., \[T_k(s)=\max(-k,\min(k,s)) \mbox{ and } G_k(s)=s-T_k(s).\]

It is worthy to mention that due to \eqref{P1} and \eqref{P3}, system \eqref{P} includes sublinear terms, for instance see \cite{BocOr1994,DuraOliv2022}.
Indeed, in \cite{BocOr1994} the authors prove  existence of a solution $u\in W^{1,2}_0(\Omega)$ for

\begin{align*}
    \label{SUBLINEAR} \tag{$S$}
\begin{cases}
-\mbox{div}\big(M(x)\nabla u\big)&= \rho(x) u \mbox{ in } \Omega\\   \ \ \ \ \ \ \ \   \ \ \ \ \ \ \ \ \ u &= 0 \mbox{ on }\Omega,
\end{cases}
\end{align*}
where $\rho\geqslant 0$ a.e. in $\Omega$, $\rho \in L^m(\Omega)$ and $m \geqslant \bigg(\dfrac{2^*}{1+\theta}\bigg)^\prime$. Later on, a singular version of \eqref{SUBLINEAR} with a $p$-Laplacian structure is addressed, see \cite{DuraOliv2022}.

Now we are in position to introduce our main contributions.

\subsection{Main results}\label{mainresults}

Our main contributions are twofold. First we consider nonlinearities which are more general and second we address the regularizing effect in the case where there is the presence of a second parameter of coupling on the nonlinearities, as it was conjectured in \cite{Durastanti}. However, as it turned out, we discover the presence of a ramification on the gain of regularity depending on the interplay between the data and the coupling parameters, below the known results for systems related to \eqref{P}, see Theorem \ref{principal} below. We point out that, in order to do that, we consider the Laplacian--like version of the Maxwell--Schr\"odinger system. We also introduce a definition which, in our view, slightly simplifies the explanation of the concept of gain of regularity.

In the literature, we say that there exist  regularizing effects in a solution of a problem or a system, whenever its regularity escapes the predicted one according to  the standard Stampacchia or Calder\'on--Zygmund theories. %order to simplify the OTo emphasize the regularity gains obtained in the solutions $u$ and $v$ of our problem, we will patent a definition that will establish conditions necessary for us to gain regularity better than that stressed by the classical Stamppachia theory.
In order to summarize this justification, we introduce the following definition on ``regularized solutions".

\begin{definition}\label{def2.1}
   Let $F \in L^m(\Omega)$ where $1\leqslant m < \frac{N}{2}$. Consider $w$ a distributional solution of
   \begin{align}\label{d2.1}
       -\mbox{div}(M(x)\nabla w)=F(x).
   \end{align}
   \begin{itemize}
       \item [a)] If $w \in L^s(\Omega)$ where $s>m^{**}$ we say $w$ is Lebesgue regularized.
       \item[b)] If $w \in W^{1,t}_0(\Omega)$ where $t> m^*$ we say that $w$ is Sobolev regularized.
   \end{itemize}
\end{definition}
    Despite that the justification of the latter definition is tacit, for the convenience of the reader we explain it. Indeed, for instance, we know by Stamppachia's classical regularity theory that, if $F\in L^m(\Omega)$ with $1\leqslant m < \frac{N}{2}$, then the distributional solution of problem \eqref{d2.1} belongs to $L^{m^{**}}(\Omega)$, for instance see \cite{Boccardo} Chapters 6 and 11. Thus, if $w\in L^s(\Omega)$ with $s>m^{**}$, then $L^s(\Omega)\hookrightarrow L^{m^{**}}(\Omega)$,  properly. In this case, we have regularizing effect for the Lebesgue summability of the solution $w$, or in short, $w$ is Lebesgue regularized. Furthermore, regarding the regularity of the gradients, there are two basic scenarios. If $1\leqslant m \leqslant (2^*)^\prime$ then once again from Stampacchia's theory if $w$ solves \eqref{d2.1} then $w\in W^{1,m^*}_0(\Omega)$. Moreover, if $(2^*)^\prime < m <\frac{N}{2}$ and if $\partial \Omega$ and $M(x)$ are sufficiently smooth, then by the Calder\'{o}n--Zygmund theory, see \cite{Chen} Chapters 5 and 10, we have   $w\in W^{1,m^*}_0(\Omega)$. Finally, remark that the  restriction $1\leqslant m < \frac{N}{2}$ is considered in order to stay away from known issues concerning the regularity of the gradients when $m>\frac{N}{2}$, for instance see \cite{Boccardo11}.

In our first theorem, we address the existence and higher regularity  for positive solutions of \eqref{P}  in the case that the summability of the source is above the threshold $(r+\theta+1)^{\prime}$.
\begin{theorem}\label{lemah}
Let $f \in L^m(\Omega)$, where $f \geqslant0$ a.e. in $\Omega$, $m\geqslant (r+\theta+1)'$, $r>1,$ and $0<\theta<\min\big(\frac{4}{N-2},1\big)$. Then there exists a  solution $(u,v)$ for \eqref{P}, with $u\in W^{1,2}_0(\Omega)\cap L^{r+\theta +1}(\Omega)$, $u\geqslant 0$ a.e. in $\Omega$ and  $v \in W^{1,2}_0(\Omega)$, $v\geqslant 0$ a.e in $\Omega$.

%\begin{align}\label{TI}
%\int_{\{|u_k|>n\}} |u_k|^{r-1}|v_k|^{\theta+1} \leqslant \frac{1}{c_1}\int_{\{|u_k|>n\}} |f|,
%\end{align}
\end{theorem}
Now, in the spirit of Definition \ref{def2.1}, we will detail the gain of regularity in Lebesgue or Sobolev spaces for the solutions of \eqref{P} given by Theorem \ref{lemah}.
%\begin{proof}
\begin{corollary}\label{C2.4}\
Let $(u,v)$ be the  solution of \eqref{P}, given by Theorem \ref{lemah}.
\begin{itemize}
\item[(A)] If $2^*<r+\theta+1 $ and $(r+\theta+1)^{\prime}\leqslant m<(2^*)^{\prime}$, then $u$ is Sobolev regularized.
%Furthermore,  we have also a regularizing effect for the  Lebesgue summability of the solution $u$, since $r+\theta+1 > 2^* > m^{**}$,  that is, $L^{r+\theta+1}(\Omega) \subset L^{2^*}(\Omega) \subset L^{m^{**}}(\Omega)$ is the continuous immersion, where $m^{**}=\frac{Nm}{N-2m}$.
\item[(B)] If $2^*<r+\theta+1$ and
    $(2^*)^{\prime} \leqslant m < \frac{N(r+ \theta + 1)}{N+2(r+\theta+1)}$
then  $u$ is Lebesgue regularized.
%In this case, we get again a regularizing effect for the Lebesgue summability of the solution
%$u$.
%\item[(C)] If $r+\theta+1 < \Big(\frac{2^*}{\theta+1}\Big)^{\prime}$ and $(r+\theta+1)^{\prime}\leqslant m<(2^*)^{\prime}$ then $v$ is Lebesgue and Sobolev regularized.
\item[(C)] If $2^*<r+\theta+1 \leqslant\frac{2^*(\theta+1)}{\theta}$ then $v$ is Sobolev regularized.
\end{itemize}
\end{corollary}
%\begin{remark}
    %We note that if $\theta< \frac{4}{N-2}$, then $\Big(\frac{2^*}{\theta +1}\Big)^{\prime}<2^*$. Thus,  the results obtained from Theorem \ref{lemah} or Corollary \ref{C2.4}, guarantee that the regularizing effect on  $v$ occurs when $r+\theta+1<2^*$, unlike the regularizing effect obtained on \ $u$, which happens when $r+\theta+1>2^*$. This is summarized in the following figure:

 %\centering
%\includegraphics[scale=0.5]{Regu.png}
%\caption{Regularizing Effect}
%\label{intWL}
%\end{remark}

Regarding Theorem \ref{lemah}, we partially address  a conjecture left in \cite{Durastanti} by R. Durastanti, where the case $m\geqslant(r+\theta+1)^\prime$ was proposed. Indeed, we have proved that the Lebesgue regularity indicated in \cite{Durastanti} is achieved  for a class of  zeroth-order nonlinear terms dominated by two variable polynomials depending on both unknowns, i.e., $\theta>0$.
Remark that, if on one hand we have considered linear differential operators with Lipschitz coefficients instead of the $p$-Laplacian, on the other hand, our nonlinear coupling satisfies properties \eqref{P1} - \eqref{P4} .

In addition, our approach allows us to investigate another regime of regularity for the source term.
Our inquiry also considers existence and regularity of solutions for the case where the summability of the source is below $(r+\theta+1)^\prime$, i.e., we address problem \eqref{P}, if the source $f$ is nonnegative and belongs to $L^m(\Omega)$ with $\left(\frac{r-\theta+1}{1-2\theta}\right)'\leqslant m<(r+\theta+1)^\prime$.

\begin{theorem}\label{principal}
Let $f \in L^m(\Omega)$, where $f \geqslant0$ a.e. in $\Omega$,  $\left(\frac{r-\theta+1}{1-2\theta}\right)'\leqslant m<(r+\theta+1)^\prime$, $r>1$ and $0<\theta<\delta$, for $\delta=\min\big(\frac{N+2}{3N-2},\frac{4}{N-2},\frac{1}{2}\big)$.
%Let $r\geqslant\frac{N+2}{N-2}$ and let $0<\theta<\frac{1}{3}$ and $f\geqslant0$ belongs to $L^m(\Omega)$ with $\left(\frac{r-\theta+1}{1-2\theta}\right)'<m<(r+\theta+1)'$.
Then there exists a solution $(u,v)$ for \eqref{P},\linebreak with $u \in W^{1,p}_0(\Omega)\cap L^{r-\theta+1}(\Omega)$, $u\geqslant 0$ a.e. in $\Omega$ and $v \in W^{1,q}_0(\Omega)$, $v\geqslant 0$ a.e. in $\Omega$, where $p = \frac{2(r-\theta+1)}{(r+\theta+1)}$ and $q=\frac{2N(1-\theta)}{N-2\theta}$. Furthermore, if  $r\geqslant\frac{N+2}{N-2}$, then $(u,v) \in W^{1,q}_0(\Omega) \times W^{1,q}_0(\Omega)$.
\end{theorem}

Once again, in the guidelines of Definition \ref{def2.1}, we now depict the regularizing effect zones guaranteed by Theorem \ref{principal}.
%olutions, for this such a test function is important to ensure , that is, to show the existence of a positive constant
%$C$ such that
%\begin{align*}
%    ||u_k||_{W^{1,2}_0} + ||v_k||_{W^{1,2}_0} + ||u_k||_{W^{1,p}_0} + ||v_k||_{W^{1,q}_0} \leqslant C.
%\end{align*}}

%\begin{proof}
%\textcolor{red}{\begin{corollary}\label{C2.4}\
%\begin{itemize}
%\item[(A)] If $r+\theta+1 >2^*$ and %$(r+\theta+1)^{\prime}<m<(2^*)^{\prime}$, then we get a regularizing effect in $u$. Furthermore,  we have also a regularizing effect for the  Lebesgue summability of the solution $u$, since $r+\theta+1 > 2^* > m^{**}$,  that is, $L^{r+\theta+1}(\Omega) \subset L^{2^*}(\Omega) \subset L^{m^{**}}(\Omega)$ is the continuous immersion, where $m^{**}=\frac{Nm}{N-2m}$.
%\item[(B)] If $r+\theta+1> 2^*$ and
 %   $(2^*)^{\prime} < m < \frac{N(r+ \theta + 1)}{N+2(r+\theta+1)}$.
%Then $r+\theta+1 > m^{**}$. In this case, we get again a regularizing effect for the Lebesgue summability of the solution
%$u$.
%\item[(C)] Let $(u,v)$ be the weak solution of \eqref{P}, given by Theorem\ref{lemah}. Then $g(x,u,v)$ belongs to $L^s(\Omega)$
%where $s=\frac{2^*(r+\theta+1)}{2^*+\theta(r+\theta+1)}$. In addition, we get
%\begin{align*}
 %   r+\theta+1 < \Big(\frac{2^*}{\theta+1}\Big)^{\prime} \iff s<(2^*)^{\prime}.
%\end{align*}
%Hence we have a regularizing effect for
%solution $v$.
%\end{itemize}
%\end{corollary}}
 %According to Definition \ref{def2.1} our next result exalts the gain  of regularity  the solutions given by the Theorem \ref{principal}.
\begin{corollary}\label{C2.6}
Let $(u,v)$ be the  solution of \eqref{P}, given by Theorem \ref{principal}
\begin{itemize}
%\item [(A)] If $r-\theta+1 > 2^*$ suppose
%that \begin{equation*} \frac{2N(r-\theta+1)}{N(r+\theta+1)+2(r-%\theta+1)} \leqslant m < (2^*)^{\prime}.
%\end{equation*} Then $u$ is Lebesgue regularized.
\item [(A)] If $r-\theta+1 > 2^*(1-\theta)$ suppose
that \begin{equation*} \Big(\frac{r-\theta+1}{1-2\theta}\Big)^{\prime} \leqslant m < \min \bigg(\frac{2N(r-\theta+1)}{N(r+\theta+1)+2(r-\theta+1)}, (r+\theta+1)^{\prime}, (2^*)^\prime \bigg).
\end{equation*} Then $u$ is Sobolev regularized.
%\item[(C)] If $2^*\geqslant r-\theta+1 >2^*(1-\theta)$ suppose that
%\begin{equation*} \Big(\frac{r-\theta+1}{1-2\theta}\Big)^{\prime} < m < \frac{2N(r-\theta+1)}{N(r+\theta+1)+2(r-\theta+1)}.
%\end{equation*}
%Then $u$ is Sobolev and Lebesgue regularized.
%\end{itemize}
\item[(B)] If $2^*(1-\theta)<r-\theta+1 \leqslant \dfrac{2^*(1-\theta)^2}{\theta}$ then $v$ is Sobolev regularized.
\end{itemize}
\end{corollary}
We will present proofs for Corollary \ref{C2.4} and Corollary \ref{C2.6}  in Section 5. For now, some remarks are in order.

\begin{remark}
%\begin{itemize}
%\item[(i)]
The intervals established for $m$ in Corollary  \ref{C2.6} are not empty. Indeed, for item (A), observe the following equivalences:
    \begin{itemize}
        \item [(i)]
        \begin{align*}
            \Big(\frac{r-\theta+1}{1-2\theta}\Big)^\prime < (2^*)^\prime \iff 2^*(1-2\theta)<r-\theta+1;
        \end{align*}
        \item[(ii)]
        \begin{align*}
              \Big(\frac{r-\theta+1}{1-2\theta}\Big)^\prime < \frac{2N(r-\theta+1)}{N(r+\theta +1)+2(r-\theta+1)} \iff 2^*(1-\theta)< r-\theta+1;
        \end{align*}
    \item[(iii)]
    \begin{align*}
              \frac{r-\theta+1}{r+\theta}=\Big(\frac{r-\theta+1}{1-2\theta}\Big)^\prime < (r+\theta+1)^\prime=\frac{r+\theta+1}{r+\theta}
              \ \ \forall \theta >0.
              \end{align*}
    \end{itemize}
Thus, in (i) and (ii), remark that we ask $r-\theta+1>2^*(1-\theta)$, while (iii) is obviously satisfied.
%\end{itemize}
%Since $  2^*(1-2\theta) <2^* < r-\theta+1 $ there follows

%Analogously, $2^*(1-\theta)<2^*<r-\theta+1$ there follows
Moreover, for item (B), observe that $2^*(1-\theta)< \dfrac{2^*(1-\theta)^2}{\theta}$, since $0<\theta<\frac{1}{2}$.

%In addition, the interval established for $m,$ in the third item of Corollary \ref{C2.6}, is also non-empty. Actually, if $r+\theta+1 \geqslant 2^*\geqslant r-\theta+1$, then there holds that $(r+\theta+1)^\prime < (2^*)^\prime$ and
%\begin{align*}
%r-\theta+1 < 2^* \iff \dfrac{N(r-\theta+1)}{N+2(r-\theta+1)} < (2^*)^\prime.
%\end{align*}
%Further,  if $2^* \geqslant r+\theta+1 > r-\theta+1$, then there holds that $(r+\theta+1)^\prime \geqslant (2^*)^\prime$ and \begin{align*}
%r-\theta+1 < 2^* \iff \dfrac{N(r-\theta+1)}{N+2(r-\theta+1)} < (2^*)^\prime.
%\end{align*}

%\item[(ii)]  We observe that for $r-\theta+1 > 2^*$, as $ r+\theta+1 > r-\theta+1$ then $r+\theta+1> 2^* \iff (r+\theta+1)^{\prime}< (2^*)^{\prime}$, in this case, the  hypothesis established in $m$ by Theorem \ref{principal}, allows us to conclude that $m<(r+\theta+1)^\prime < (2^*)^{\prime}$. This means that, if we take $r-\theta+1>2^*$, the intervals determined for $m$ in the first two items of the Corollary \ref{C2.6} are non-empty.
\end{remark}

Finally, let us stress that, inspired by the classical approach of the school of G. Stampacchia, L. Boccardo, among others, see \cite{Boccardo1,Boccardo0,Boccardo2,Stam1, Stam2} and the references therein, the main ingredient for our results is based on a carefully choice of tailored  test functions. Indeed, by means of subtle  modifications on the test functions  we are able to address the regimes of regularity described in Theorems \ref{lemah} and \ref{principal}, see Lemmas \ref{teo47} and \ref{lem4.10}.
After that, we follow the standard approach of determining a priori estimates for solutions of an approximate problem and then passing to the limit.

\section{Preliminaries}
In this section we provide certain technical results used in the present paper. Despite that not all of them are new, for the convenience of the reader, we give some proofs.
We begin by defining for $\tau>0$  the following truncations
%\begin{align*}
%h_k(x,t,s) =
%\begin{cases}
% h(x,t,s) \ \ \mbox{if} \ \ |h(x,t,s)|\leqslant k; \\
% k\ {\rm sgn}(h(x,t,s))\ \ \mbox{if} \ \ %|h(x,t,s)|>k;
%\end{cases}
%\end{align*}
%\begin{align*}
%g_k(x,t,s) =
%\begin{cases}
%g(x,t,s) \ \ \mbox{if} \ \ |g(x,t,s)|\leqslant k; \\
%k\ {\rm sgn}(g(x,t,s))\ \ \mbox{if} \ \ |g(x,t,s)|>k.
%\end{cases}
%\end{align*}
%
\begin{align}
\label{truncationshg}
\begin{cases}
g_\tau(x,t,s)=T_\tau \big(g(x,t,s)\big); \\
h_\tau(x,t,s)=T_\tau \big(h(x,t,s)\big).		
\end{cases}
\end{align}
Moreover,  it is clear that $g_\tau$ and $h_\tau$ satisfy the hypotheses \eqref{P22} and \eqref{P4} respectively.\\
%where $F_k\in L^\infty(\Omega),$ $|F_k|\leqslant|f|,$ $\forall\ n\in\mathbb{N}$ and $F_k\to f$ in $L^m(\Omega)$.\\

As a first step, we provide certain convergence results which will be employed in the proofs of our main results.  As it will be clear, by cropping certain technical details from  Theorem \ref{principal} and \ref{lemah} we simplify their proofs. We just mention that the difference between cases $(ii)$ and $(iii)$ in the next lemma, comes from the fact that, naturally,  when $f$ is less regular, the estimates on the mixed terms become also less regular, see Lemmas \ref{teo47} and \ref{lem4.10}.
\begin{lemma}\label{lem48}
Let $f\in L^m(\Omega)$ with $m\geqslant1$, $\{u_k\}$ bounded in $W^{1,s}_0(\Omega)$ and $\{v_k\}$ bounded in $W^{1,t}_0(\Omega)$, where  $s\geqslant t \geqslant \frac{N(\theta+1)}{N+\theta+1}$ and $h$, $g$ be two Carathéodory functions satisfying \eqref{P1} and \eqref{P3}. Then, there exist $u$ and $v$ in $W^{1,t}_0(\Omega)$ such that, up to subsequences relabeled the same:

\begin{enumerate}[label=\textit{(\roman*)}]
\item \label{lem48:1} If $\int_{\{|u_k|>n\}}|u_k|^{r-1}|v_k|^{\theta+1} \leqslant C\int_{\{|u_k|>n\}}|f|,$

then $g(x,u_k,v_k) \to g(x,u,v) \ \ {\rm in} \ L^1(\Omega)$.

\item \label{lem48:2} If $\int_{\{|v_k|>n\}}|u_k|^r|v_k|^{\theta+1}\leq C$ and  $\ \{|u_k|^r\} \ \mbox{is uniformly integrable}$

then $ h(x,u_k,v_k) \to h(x,u,v) \ \ {\rm in} \ L^1(\Omega)$.
\item \label{lem48:3} If $0<\theta < \frac{1}{2}$ and $\int_{\{|v_k|>n\}}|u_k|^r|v_k|^{1-\theta}\leq C$ and  $\ \{|u_k|^r\} \ \mbox{is uniformly integrable}$

then $ h(x,u_k,v_k) \to h(x,u,v) \ \ {\rm in} \ L^1(\Omega)$.
\end{enumerate}

\end{lemma}

%\begin{lemma}\label{lem48}
%Let $f\in L^m(\Omega)$ with $m\geqslant1$, $\{u_k\}$ and $\{v_k\}$ bounded in $W^{1,p}_0(\Omega)$, $p>1$ and $h$, $g$ be two Carathéodory functions satisfying \eqref{P1} and \eqref{P3}.
%Moreover,  suppose that
%\begin{enumerate}[label=\textit{(\roman*)}]
%\item \label{lem48:1}
%$\int_{\{|u_k|>n\}}|u_k|^r|v_k|^{\theta+1} + \int_{\{|u_k|>n\}}|u_k|^{r-1}|v_k|^{\theta+1} \leqslant C\int_{\{|u_k|>n\}}|f|;$
%\item \label{lem48:2} $\ |u_k|^r \ \mbox{is uniformly integrable.}$
%\end{enumerate}
%Then, there exist $u$ and $v$ in $W^{1,p}_0(\Omega)$ such that, up to subsequences relabeled the same
%\begin{align*}
%h(x,u_k,v_k) \to h(x,u,v) \ \ {\rm in} \ L^1(\Omega)
%\end{align*}
%and
%\begin{align*}
%g(x,u_k,v_k) \to g(x,u,v) \ \ {\rm in} \ L^1(\Omega).
%\end{align*}
%\end{lemma}

\begin{proof}
$(i)$
Since $\frac{N(\theta+1)}{N+\theta+1}\leqslant t< s$, then $\theta+1\leqslant t^*< s^*$, and also $u, v \in W^{1,t}(\Omega)$ such that,
up to subsequences
%by hypothesis $\{u_k\}$ and $\{v_k\}$ are bounded in $W^{1,p}_0(\Omega),$ there exist subsequences still indexed by $u_k$ and $v_k$ and functions $u$ and $v\in W^{1,p}_0(\Omega)$ such that
\begin{align}
\label{con.ast}
\begin{cases}
&u_k \rightharpoonup  u\ \ \mbox{weakly in} \ \ W^{1,s}_0(\Omega)\\
&v_k \rightharpoonup v\ \ \mbox{weakly in} \ \ W^{1,t}_0(\Omega),\\
&u_k \to u\ \ \mbox{ in} \ L^{t^*}(\Omega),\ \mbox{and a.e in}\
\Omega\\
&v_k \to v\ \ \mbox{ in} \ L^{t^*}(\Omega),\ \mbox{and a.e in}\ \Omega.
\end{cases}
\end{align}
Of course, we also have $g(x,u_k,v_k) \to g(x,u,v) \ \ \mbox{a.e in} \ \ \Omega$. Further, from \eqref{P1} and \ref{lem48:1}, given $E\subset\Omega$,  a measurable set, there follows that
\begin{align*}
\int_{E} |g(x,u_k,v_k)| &\leqslant c_2 \int_{E} |u_k|^{r-1}|v_k|^{\theta+1}\\
&= c_2 \int_{E\cap\{|u_k|\leqslant n\}} |u_k|^{r-1}|v_k|^{\theta+1} + c_2 \int_{E\cap\{|u_k|> n\}} |u_k|^{r-1}|v_k|^{\theta+1}\\
&\leqslant c_2 n^{r-1}\int_{E} |v_k|^{\theta+1} + c_2 \int_{\{|u_k|> n\}} |u_k|^{r-1}|v_k|^{\theta+1}\\
&\leqslant c_2 n^{r-1}\int_{E} |v_k|^{\theta+1} + C\int_{\{|u_k|>n\}} |f|.
\end{align*}
 \hspace{-.4cm}\textbf{Claim 1.} Given $\sigma>0$, there exist $n_0\in\mathbb{N}$ and $\delta>0$ with ${\rm \mbox{meas}}(E)<\delta$, such that
\begin{align*}
C\int_{\{|u_k|>n\}} |f| \leqslant \frac{\sigma}{2} \ \ \mbox{and} \ \ n^{r-1}\int_{E} |v_k|^{\theta+1} \leqslant \frac{\sigma}{2}.
\end{align*}
Note that, by H\"{o}lder's inequality
\begin{align*}
C\int_{\{|u_k|>n\}} |f| \leqslant C \|f\|_{L^m}\cdot{\rm \mbox{meas}}(\{|u_k|>n\})^{\frac{1}{m^\prime}}.
\end{align*}
Moreover, by considering $C_0$ such that
\begin{align*}
C_0 > \int_{\Omega} |u_k|^{t^*} \geqslant \int_{\{|u_k|>n\}} |u_k|^{t^*} > \int_{\{|u_k|>n\}} n^{t^*} = n^{t^*} {\rm \mbox{meas}}(\{|u_k|>n\})
\end{align*}
thence
\begin{align*}
{\rm \mbox{meas}}(\{|u_k|>n\}) < \frac{C_0}{n^{t^*}} \to 0, \ {\rm when}\ n\to\infty.
\end{align*}
That is, for all $\sigma_1>0$, there exists $n_0\in\mathbb{N}$ such that for $n>n_0$, we have
\begin{align*}
{\rm \mbox{meas}}(\{|u_k|>n\}) \leqslant \sigma_1 \ \forall\ k\in\mathbb{N}.
\end{align*}
Thus, by taking $\sigma_1=\frac{\sigma}{2C\left(\|f\|_{L^m}+1\right)}$, for $n>n_0$ fixed we have that
\begin{align*}
C\int_{\{|u_k|>n_0\}} |f| \leqslant \frac{\sigma}{2}.
\end{align*}
On the other hand, as $\theta+1<t^*$, there follows that
\begin{align*}
\|v_k-v\|^{\theta+1}_{L^{\theta+1}} \leqslant C \|v_k-v\|^{\theta+1}_{L^{t^*}} \to 0, \ \mbox{when} \ k\to\infty.
\end{align*}
In this way, by the Vitali Theorem there exists $\sigma_2$ such that ${\rm \mbox{meas}}(E)<\delta$ implies that
\begin{align*}
\int_{E}|v_k|^{\theta+1} < n^{r-1}\sigma_2.
\end{align*}
In this fashion, by taking $\sigma_2=\frac{\sigma}{2n^{r-1}}$, we get
\begin{align*}
\int_{E}|v_k|^{\theta+1} < \frac{\sigma}{2},  \ \mbox{ proving our claim}.
\end{align*}
At this point, let us stress that by {\bf Claim 1}, we have
\begin{align*}
\exists\ \delta>0; \ \ \ {\rm \mbox{meas}} (E)<\delta \ \ \mbox{implies} \ \ \int_{E}|g(x,u_k,v_k)|<\sigma, \ \ \forall \ \sigma>0,
\end{align*}
so that, by the Vitali Theorem, we get
$g(x,u_k,v_k) \to g(x,u,v) \ \ {\rm in} \ L^1(\Omega)$.

\ref{lem48:2} Now, remark that by \eqref{con.ast} we have
$ h(x,u_k,v_k) \to h(x,u,v) \ \ \mbox{ a.e in} \ \ \Omega$, it up to subsequences relabeled the same. Moreover, for $\sigma>0$, given $E\subset \Omega,$ by hypothesis \eqref{P3} combined with assumption \ref{lem48:2} p. \pageref{lem48:2} we have that
\begin{align*}
\int_{E}|h(x,u_k,v_k)| &\leqslant d_2 \int_{E} |u_k|^r |v_k|^\theta\\
&= d_2 \int_{E\cap \{|v_k|\leqslant n\}} |u_k|^r |v_k|^\theta + d_2 \int_{E\cap \{|v_k|> n\}} |u_k|^r |v_k|^{\theta+1-1}\\
&\leqslant d_2 n^\theta \int_{E} |u_k|^r + \frac{d_2}{n} \int_{\{|v_k|>n\}} |u_k|^r |v_k|^{\theta+1}\\
&\leqslant d_2 n^\theta \int_{E} |u_k|^r + \frac{d_2 C}{n},
\end{align*}
for all $n\in\mathbb{N}$. In particular, by taking $n_0$ such that $\frac{C}{n_0}<\frac{\sigma}{2d_2}$, we arrive at
\begin{align*}
\int_{E} |h(x,u_k,v_k)| \leqslant d_2 n_0^\theta \int_{E} |u_k|^r +\frac{\sigma}{2}.
\end{align*}
However, since $\{|u_k|^r\}$ is by hypothesis uniformly integrable, there exist $\delta>0$ for which, if ${\rm \mbox{meas}}(E)<\delta$, one has
\begin{align*}
\int_{E} |u_k|^r < \frac{\sigma}{2d_2 n_0^\theta}, \ \ \forall \ k,n \in\mathbb{N},
\end{align*}
and then
\begin{align*}
\int_{E}|h(x,u_k,v_k)| < \sigma,
\end{align*}
so that $h(x,u_k,v_k) \to h(x,u,v)$ in $L^1(\Omega)$, by the Vitali Theorem.

\ref{lem48:3} This proof is very similar to the last one, nevertheless, in the current case, the argument for the second convergence, is more delicate since now we lost some regularity of the estimates of the mixed terms. In any case, observe that for $\gamma=1-2\theta$, since $\int_{\{|v_k|>n\}}|u_k|^r|v_k|^{1-\theta}\leq C$ one has that
\begin{align*}
\int_{E}|h(x,u_k,v_k)| &\leqslant d_2 \int_{E} |u_k|^r |v_k|^\theta\\
&= d_2 \int_{E\cap \{|v_k|\leqslant n\}} |u_k|^r |v_k|^\theta + d_2 \int_{E\cap \{|v_k|> n\}} |u_k|^r |v_k|^{\theta+\gamma-\gamma}\\
&\leqslant d_2 n^\theta \int_{E} |u_k|^r + \frac{d_2}{n^\gamma} \int_{\{|v_k|>n\}} |u_k|^r |v_k|^{\theta+\gamma}\\
&\leqslant d_2 n^\theta \int_{E} |u_k|^r + \frac{d_2 C}{n^\gamma}, \ \ \  \mbox{for all} \ n \in \mathbb{N},
\end{align*}
 so that by repeating the same argument as in \ref{lem48:2}  we prove that $\{h(x,u_k,v_k)\}$ is uniformly integrable and the result follows once more by the Vitali Theorem.

\end{proof}

In the next lemma, despite its technical nature, we provide the mathematical background in order to validate the use of very specific test functions related to the proof of  Theorem \ref{principal}, see Lemma \ref{lem4.10}.
\begin{lemma}\label{teclema}
	Let $u,v\in W^{1,2}_0(\Omega)  \cap L^\infty(\Omega)$, $u\geqslant 0$ a.e. in $\Omega$,  $v >0 $ a.e. in $\Omega$, and let $\varphi_\varepsilon(t)=(t+\varepsilon)^\gamma-\varepsilon^\gamma$ be a Lipschitz function for all $t>0$, where $\varepsilon>0$, $0<\gamma<1$ and $H:\Omega\times\mathbb{R}\times\mathbb{R}\to\mathbb{R}$ be a Carathéodory function satisfying
 \[|H(x,t,s)|\leqslant C|t|^{\sigma_1}|s|^{\sigma_2}, \mbox{ for } \sigma_i>0, i=1,2.\]

 Then
\begin{enumerate}[label=\textit{(\alph*)}]
\item\label{tec:i} $\gamma\int_{\Omega_+}|\nabla u|^2 u^{\gamma-1} \leqslant \lim_{\varepsilon\to0^{+}} \int_\Omega \nabla u\cdot \nabla\varphi_\varepsilon(u)$,
%\item\label{tec:ii} $\gamma\int_{\Omega}|\nabla v|^2 v^{\gamma-1} \leqslant \lim_{\varepsilon\to0^{+}} \int_\Omega \nabla v\cdot \nabla\varphi_\varepsilon(v)$,

\item\label{tec:ii} $\int_\Omega H(x,u,v)u^\gamma = \lim_{\varepsilon\to0^{+}}\int_\Omega H(x,u,v)\varphi_\varepsilon(u)$,
\end{enumerate}
where %\varphi_\varepsilon=\varphi_\varepsilon(u)$ and
$\Omega_+ = \{x\in \Omega: u > 0\}$.
\end{lemma}
\begin{proof}
\ref{tec:i} Remark that $u(x)>0$ a.e. in $\Omega_+$, so that by setting
\begin{align*}
\omega_u =
    \begin{cases}
        u^{\gamma-1} \ \ \mbox{a.e. in} \ \Omega_+ ;\\
        + \infty \ \ \mbox{a.e. in} \ \Omega \backslash \Omega_+,
    \end{cases}
\end{align*}
$\omega_u$ is measurable and well--defined. Moreover, since $|\nabla u|= 0$ a.e. in $\Omega\backslash \Omega_+$, by using the real extended line arithmetic rules, we have $|\nabla u|^2 \omega_u = 0 \ \mbox{a.e. in} \ \Omega\backslash \Omega_+$. Thus, \[\int_{\Omega} |\nabla u|^2\omega_u = \int_{\Omega_+} |\nabla u|^2 u^{\gamma}.\]

Now, observe that $0\leqslant (u+\varepsilon)^{\gamma-1}<u^{\gamma-1}$
and $(u+\varepsilon)^{\gamma-1} \to \omega_u \ \ \mbox{a.e. in} \ \ \Omega \ \ \mbox{when} \ \ \varepsilon\to0^+$, there follows from the Fatou Lemma that
\begin{align*}
\gamma \int_{\Omega_+} |\nabla u|^2 u^{\gamma-1} &\leqslant \liminf_{\varepsilon\to0^+} \gamma \int_\Omega |\nabla u|^2 (u+\varepsilon)^{\gamma-1}\\
&= \liminf_{\varepsilon\to0^+} \int_\Omega \nabla u\cdot \nabla\varphi_\varepsilon(u).
\end{align*}\\
%In addition, it is clear that  $\gamma\int_\Omega|\nabla v|^2 v^{\gamma-1} \leqslant \lim_{\varepsilon\to 0^{+}} \int_\Omega |\nabla v|^2(v+\varepsilon)^{\gamma-1}$, since $v>0$ a.e. in $\Omega$.\\
\ref{tec:ii} Note that
\begin{align*}
|H(x,u,v)[(u+\varepsilon)^\gamma-\varepsilon^\gamma]| &\leqslant |H(x,u,v)| [(u+\varepsilon)^\gamma + \varepsilon^\gamma]\\
&\leqslant c_2 |u|^{\sigma_1} |v|^{\sigma_2} [(u+\varepsilon)^\gamma + \varepsilon^\gamma].	
\end{align*}
	Suppose $0<\varepsilon\leqslant1,$ we have
	\begin{align*}
		|H(x,u,v)[(u+\varepsilon)^\gamma-\varepsilon^\gamma]| \leqslant C\|u\|^{\sigma_1}_{L^\infty}\|v\|^{\sigma_2}_{L^\infty} [(\|u\|_{L^\infty}+1)+1].
	\end{align*}
	Moreover,
	\begin{align*}
		H(x,u,v)[(u+\varepsilon)^\gamma -\varepsilon^\gamma] \to H(x,u,v)u^\gamma, \ \ \mbox{a.e. in}\ \ \Omega, \ \mbox{when} \ \ \varepsilon\to0^+ .
	\end{align*}
Thus, by the Lebesgue Dominated Convergence Theorem
	\begin{align*}
		\int_\Omega H(x,u,v)u^\gamma = \lim_{\varepsilon\to0^+}\int_\Omega H(x,u,v)\varphi_\varepsilon(u).
	\end{align*}
\end{proof}

\begin{remark}\label{1029}
It is clear that by arguing in an analogous manner to the proof of Lemma, since $v>0$ a.e. in $\Omega$ we also have the validity of
\begin{enumerate}[label=\textit{(\alph*)}]
\item\label{1603} $\gamma\int_{\Omega}|\nabla v|^2 v^{\gamma-1} \leqslant \lim_{\varepsilon\to0^{+}} \int_\Omega \nabla v\cdot \nabla\varphi_\varepsilon(v)$,
\item\label{1608} $\int_\Omega H(x,u,v)v^\gamma = \lim_{\varepsilon\to0^{+}}\int_\Omega H(x,u,v)\varphi_\varepsilon(v)$.
\end{enumerate}
\end{remark}

Now, by using hypotheses \eqref{P1}-\eqref{P5}, and a standard argument based on the Schauder Fixed Point Theorem, we obtain
 existence of solutions to preliminary version of \eqref{P}. For the sake of simplicity, its proof will be omitted, we only point out that in order to construct a well-defined operator whose fixed points are weak solutions, we use \eqref{P2}.
\begin{proposition}\label{prop1}
Let $\Phi\in L^\infty(\Omega)$ and $\tau>0$. There exists a weak solution $(u,v) \in  W^{1,2}_0(\Omega)\times W^{1,2}_0(\Omega)$ to the system
\begin{align*}\tag{$P_{A}$}\label{P_A}
\begin{cases}
\ -\mbox{div}(M(x)\nabla u)+g_\tau(x,u,v)=\Phi\\ \ -\mbox{div}(M(x)\nabla v)=h_\tau(x,u,v)+\frac{1}{\tau},
\\
			\ \ \ \ \ u=v=0 \ \ \mbox{on} \ \ \partial \Omega.
\end{cases}
\end{align*}
Furthermore, $u, v \in L^\infty(\Omega)$.
\end{proposition}
\begin{proof}
Fix $\zeta\in L^{2}(\Omega).$ By the Leray--Lion Theorem, see \cite{Boccardo}, Theorem 5.6, p. 40, there exists $u\in W^{1,2}_0(\Omega)$ such that
\begin{align}\label{I}
\int_{\Omega}M(x)\nabla u\cdot\nabla\varphi  + \int_{\Omega} g_{\tau}(x,u,\zeta)\varphi = \int_{\Omega} \Phi\varphi \ \ \ \ \forall\ \varphi \in W^{1,2}_{0}(\Omega).
\end{align}
It is clear that $u$ is unique, so that we define $S:L^{2}(\Omega)\to L^{2}(\Omega)$, where $S(\zeta)=u$ is given as above.
Further, for $u=S(\zeta)$ it is straightforward to that
\begin{align}\label{D2}
\|u\|_{W^{1,2}_0}\leqslant R_1 ,
\end{align}
for $R_1=C\left[\|\Phi\|_{L^\infty}+\tau\right]{\rm meas}(\Omega)^{1/2}$ .\\

In an analogous manner, by the same result, given $w\in L^{2}(\Omega)$ fixed, there exists $\eta \in W^{1,2}_0(\Omega)$ satisfying
\begin{align*}
\int_\Omega M(x) \nabla \eta\cdot\nabla\psi =\int_\Omega \Big(h_{\tau}(x,w,\eta\big)+\frac{1}{\tau}\Big)\psi \ \ \forall \ \psi\in W^{1,2}_0(\Omega),
\end{align*}
and once again, since $\eta$ is unique
we define $T:L^{2}(\Omega)\to L^{2}(\Omega)$, where $T(w)=\eta$.
Let us stress that for $\eta=T(w)$, we have
\begin{align}\label{D3}
\|\eta\|_{W^{1,2}_0}\leqslant R_2,
\end{align}\
where $R_2= C\Big(\frac{\tau^2+1}{\alpha\tau}\Big) \mbox{meas}(\Omega)^{1/2}$.

Now, let us define $G:L^{2}(\Omega)\to L^{2}(\Omega)$, where $G=T \circ S$. First observe that $G$ is obviously continuous.

We claim that $G$ satisfies the assumptions of the Schauder Fixed Point Theorem. Indeed, consider
\[B=B(0,R)=\{u\in L^2(\Omega); \ \|u\|_{L^2}\leqslant\ R\}\]
where $R = C_1\max \{R_1,R_2\}$ and $C_1>0$ is the Poincar\'e constant. Note that, by \eqref{D2} and \eqref{D3}, we have $G((B))\subset B$.

Finally, remark that $G(B)\subset W^{1,2}_0(\Omega)$ so that by the Rellich--Kondrachov Theorem, $G$ is also compact.
Therefore, by the Schauder Fixed Point Theorem, there exists $v\in G(B)\subset W^{1,2}_0(\Omega)$ such that
\begin{align*}
v=G(v)=T(S(v))=T(u).
\end{align*}
\end{proof}
\section{Approximate Problem}
This section comprehends the heart of the contributions on the present paper.  Indeed, we explore the choice of tailored test functions for a favorable  approximate version of \eqref{P} in order to obtain certain key estimates, see Lemmas \ref{teo47} and \ref{lem4.10}.

Nevertheless, we start by obtaining estimates for a first version of approximate problem, i.e.,  that the solutions given  by Proposition \ref{prop1} are bounded in $W^{1,2}_0(\Omega) \cap L^{\infty}(\Omega)$. We stress that, surprisingly, in order  to obtain $L^{\infty}$ estimates for $v$, we have to impose that $\theta<\frac{4}{N-2}$.% that is, for $N\leqslant 6$ we have to $\theta$ is any in between $0$ and $1$. Otherwise $\theta$ depends on the dimension, in other words, the larger the dimension, smaller is the value of $\theta$.
%\begin{align*}
%&\int_\Omega \nabla u\nabla \varphi +\int_\Omega h_k(x,u,v)\varphi = \int_\Omega \Phi \varphi \ \  \forall \ \varphi \in W^{1,2}_0(\Omega),\\
%&\int_\Omega \nabla v\nabla \psi = \int_\Omega g_k(x,u,v)\psi \ \  \forall \ \psi \in W^{1,2}_0(\Omega)
%\end{align*}

\begin{lemma} \label{lemma31} Let $\Phi \in L^{m}(\Omega)$ with $m\geqslant 1$  and let $(u,v)$ be the solution of system \eqref{P_A} given by Proposition \ref{prop1}. Then there exists a constant $C>0$ which does not depend on $\tau$, such that
\begin{itemize}
    \item [(i)] if $\Phi \in L^{m}(\Omega)$ for  $m\geqslant (2^*)'$ then
    \begin{align*}
    \|u\|_{W^{1,2}_0} + \|v\|_{W^{1,2}_0}\leqslant C\|\Phi \|_{L^m};
\end{align*}
 \item [(ii)] if $\Phi \in L^{m}(\Omega)$ for  $m>\frac{N}{2}$ and $0<\theta<\min \big (\frac{4}{N-2},1\big)$  then
 \begin{align*}
\|u\|_{L^{\infty}}+\|v\|_{L^{\infty}}\leqslant C\|\Phi \|_{L^m}.
\end{align*}
\end{itemize}
\end{lemma}
\begin{proof}

By taking $\varphi=u$  in the weak formulation of the first equation of \eqref{P_A}, and by combining the ellipticity of $M$ with H\"{o}lder's and Sobolev's inequalities we end up with
\begin{align}
\nonumber
\|u\|_{L^{2^*}} &\leqslant C \|\Phi\|_{L^m},\\
\label{e4.3}
\|\nabla u\| _{L^2}&\leqslant C \|\Phi\|_{L^m} \mbox{ and }\\
\nonumber
\displaystyle \int_\Omega g_\tau(x,u,v)u
&\leqslant C\|\Phi\|^2_{L^m}.
\end{align}
%and by combining the above estimate with \eqref{grd4.2}, we get
%\begin{align} \label{et4.5}
%
%\end{align}
%On the other hand, by \eqref{e.4.2} and \eqref{e4.3} there follows
%\begin{align}\label{linf5}
%
%\end{align}
Further, we claim that
\begin{align}\label{claim1}
\int_{\Omega}|u|^r|v|^{\theta+1}\leqslant C\|\Phi\|_{L^m}.	
\end{align}
In fact, by \eqref{P22} and the definition of $g_\tau$, see \eqref{truncationshg}, it is clear that
$\int_\Omega g_\tau(x,u,v)u
\geqslant \int_{\{|g|\leqslant \tau\}} g(x,u,v)u,
$
and thus, from \eqref{e4.3}, we get
\begin{align*}%\label{ast}
\int_{\{|g|\leqslant \tau\}} g(x,u,v)u \leqslant C\|\Phi\|^2_{L^m}.
\end{align*}
%Now, since $h_k$ is  Carathéodory, % function a.e. in $\Omega$, so
%\begin{align*}
%{\rm \mbox{meas}}(\widetilde{\Omega}) = {\rm \mbox{meas}}(\{x\in \Omega;   h \ \mbox{is not finite}\})=0,
%\end{align*}
%Thus for all $\varepsilon>0$ and $x_\circ\in \Omega$\textbackslash$\widetilde{\Omega}$ there exist $k_\circ=k_\circ(x_\circ)$ tal que
%\begin{align*}
%k>k_\circ \ \ \implies \ \ |\mathcal{X}_{\{|h|\leqslant k\}}(x_\circ)-\mathcal{X}_\Omega(x_\circ)|= 0 < \epsilon,
%\end{align*}
%i.e.,
%\begin{align*}
%\mathcal{X}_{\{|h|\leqslant k\}}(x_\circ) \to \mathcal{X}_\Omega(x_\circ) \ \ \forall x_\circ\in \Omega\mbox{\textbackslash}\widetilde{\Omega}.
%\end{align*}
%Since
%\begin{align*}
%\int_{\Omega}h(x,u,v)u\mathcal{X}_{\{|h|\leqslant k\}} = \int_{\{|h|\leqslant k\}} h(x,u,v)u \geqslant 0
%\end{align*}
%and
%\begin{align*}
%\int_{\{|h|\leqslant k\}} h(x,u,v)u < \infty.
%\end{align*}
In this way, if we recall \eqref{P1}, by taking $\tau\to+\infty$ in the latter inequality, as a direct application of Fatou's lemma we arrive at
\begin{align*}%\label{ast2}
 \int_\Omega	|u|^r|v|^{\theta+1}\leqslant \int_\Omega g(x,u,v)u \leqslant C \|\Phi\|^2_{L^m},
\end{align*}
proving our claim.

Further, by taking $\psi = v$ in the weak formulation of the second equation of \eqref{P_A},
%\begin{align*}
%\int_{\Omega} M(x)\nabla v \cdot\nabla v = %\int_{\Omega}g(x,u,v)v,
%\end{align*}
by combining \eqref{P3} and the ellipticity of $M$, we have
\begin{align*}
    \alpha \int _{\Omega} |\nabla v|^2 \leqslant \int_\Omega \Big(h_{\tau}(x,u,v)+\frac{1}{\tau}\Big)v &\leqslant \int_\Omega |h_{\tau}(x,u,v)| |v| + \frac{1}{\tau}\int_\Omega |v| \\ &\leqslant d_1 \int_{\Omega} |u|^r|v|^{\theta+1}+ \frac{C}{\tau}||v||_{W^{1,2}_0}.
\end{align*}
Thus, taking $\tau \to +\infty$ and by \eqref{claim1} there follows that $
     \int _{\Omega} |\nabla v|^2 \leqslant C \|\Phi\|^2_{L^m}$
and hence, by combining \eqref{e4.3} and the last inequality, we get
\begin{align*}
    \int_{\Omega} |\nabla u|^2 + \int_{\Omega}|\nabla v|^2 \leqslant C \|\Phi\|^2_{L^m},
\end{align*}
what proves $(i)$.

Now we proceed to the $L^{\infty}(\Omega)$ estimates. For this,  let us recall the defintion of one the standard Stampacchia's truncation, $G_k(s)=(|s|-k)^+\mbox{sign(s)}$. Then, we take $\varphi=G_k(u)$ in the weak formulation of the first equation of \eqref{P_A}, obtaining
\begin{equation}\label{4.8est}
 \int_ \Omega M(x) \nabla u \cdot \nabla u G_{k}'(u) +\int_ \Omega g_\tau(x,u,v)G_k(u) = \int_ \Omega \Phi G_k(u).	
\end{equation}
In addition, notice that clearly, there holds
\begin{align*}%\label{linf2}
\nonumber\int_\Omega g_\tau(x,u,v)G_k(u) &= \nonumber\int_{\{|u|>k\}} g_\tau(x,u,v)G_k(u) \\
%\nonumber& = \int_{\{u>k\}} h_k(x,u,v)G_k(u) + %\int_{\{u <-k\}} h_k(x,u,v)G_k(u)\\
& = \int_{\{u>k\}} g_\tau(x,u,v)(u-k) + \int_{\{u <-k\}} g_\tau(x,u,v)(u+k).
\end{align*}
Moreover, as a straightforward consequence of \eqref{truncationshg} and the definition of $G_k(.)$, we  have
\begin{equation}
\nonumber
%\label{linf3}
\int_{\Omega} g_\tau(x,u,v)G_k(u)\geqslant0.
\end{equation}
Thus, from the latter inequality, by using the ellipticity of $M$ and Hölder's inequality on the right-hand side of \eqref{4.8est}, we get
\begin{align*}
\alpha\int_\Omega |\nabla G_k(u)|^2 \leqslant \int_\Omega \Phi G_k(u)\leqslant \Big(\int_{A^u_K}|\Phi|^\frac{2N}{N+2}\Big)^{\frac{N+2}{2N}} \Big( \int_{\Omega}|G_k(u)|^{2^*}\Big)^\frac{1}{2^*},
\end{align*}
where $A^u_k = \{|u|>k\}$.

Additionally, recall that by Sobolev's and Hölder's inequalities there follows
\begin{align*}
    \Big( \int_{\Omega} |G_k(u)|^{2^*}\Big)^{\frac{2}{2^*}} \leqslant \Big(\int_{\Omega} |G_k(u)|^{2^*}\Big)^\frac{1}{2^*}\|\Phi\|_{L^m} \mbox{meas}(A^u_k)^{[1-\frac{2N}{(N+2)m}]\frac{N+2}{2N}},
\end{align*}
so that by the latter inequalities we arrive at
\begin{align}\label{G4.12}
\Big( \int_{\Omega} |G_k(u)|^{2^*}\Big)^{\frac{1}{2^*}}\leqslant C |\Phi\|_{L^m} \mbox{meas}(A^u_k)^{[1-\frac{2N}{(N+2)m}]\frac{N+2}{2N}}.
\end{align}
Moreover, by Hölder's inequality and \eqref{G4.12}, we have
\begin{align*}
    \int_{\Omega}|G_k(u)| = \int_ {A^u_k}|G_k(u)| \leqslant \mbox{meas}(A^u_k)^{\frac{N+2}{2N}}\Big(\int_{\Omega} |G_k(u)|^{2^*}\Big)^{\frac{1}{2^*}}
    \leqslant C \|\Phi\|_{L^m} \mbox{meas} (A^u_k)^{\alpha}.
\end{align*}
Where $\alpha = 1+ \frac{2}{N}-\frac{1}{m}> 1$, since $m>\frac{N}{2}$. Hence, by Lemma $6.2$ in \cite{Boccardo} p. 49,
\begin{align*}
u\in L^\infty(\Omega) \ \mbox{and} \ \|u\|_{L^\infty}\leqslant C\|\Phi\|_{L^m},
\end{align*}
where we stress that the restriction of $\theta$ was not used.

Finally,  we handle the $L^\infty(\Omega)$ estimates on $v$. Indeed, if we take $\psi=G_k(v)$ in the second equation of \eqref{P_A}, by combining \eqref{P3}, the ellipticity of M and the \begin{align*}
\alpha\int_\Omega|\nabla G_k(v)|^2 &\leqslant \int_{\Omega} \Big( h_{\tau}(x,u,v) + \frac{1}{\tau}\Big)G_k(v) \leqslant \int_{A^v_k}|u|^r|v|^{\theta} |G_k(v)|+ \frac{1}{\tau}\int_\Omega |G_k(v)| \\
&\leqslant d_2 \|u\|^r_{L^{\infty}} \int_{A^v_k}|v|^{\theta} |G_k(v)| + \frac{2C}{\tau}||v||_{W^{1,2}_0} \mbox{meas}(\Omega), \ \mbox{where} \ A^v_k = {\{|v|>k \}}.
\end{align*}
Taking $\tau \to + \infty$ we have
\[\alpha\int_\Omega|\nabla G_k(v)|^2 \leqslant d_2 \|u\|^r_{L^{\infty}} \int_{A^v_k}|v|^{\theta} |G_k(v)|. \]
Further, by a straightforward combination of Sobolev's and Hölder's inequalities we get
\begin{align*}
\alpha S^2\left(\int_\Omega|G_k(v)|^{2^*}\right)^\frac{2}{2^*} \leqslant d_2 \|u\|^r_{L^{\infty}}\Big(\int_{A^v_k}|v|^{(2^*)'\theta}\Big)^{\frac{1}{(2^*)'}}\left(\int_\Omega|G_k(v)|^{2^*}\right)^{\frac{1}{2^*}},
\end{align*}
so that
\begin{align*}
S^2\left(\int_\Omega|G_k(v)|^{2^*}\right)^\frac{1}{2^*} \leqslant \Big(\int_{A^v_k}|v|^{(2^*)'\theta}\Big)^{\frac{1}{(2^*)'}},
\end{align*}
and then, by applying once again the Hölder inequality on the right-hand side, with the exponents $\frac{2^*}{(2^*)'\theta}$ and $\frac{2^*}{2^* - \theta(2^*)'}$, we arrive at
\begin{align*}%\label{est17}
\nonumber\left(\int_\Omega|G_k(v)|^{2^*}\right)^\frac{1}{2^*}
&\leqslant  C \mbox{meas}(A^v_k)^{\Big[\frac{2^* -\theta (2^*)'}{2^*}\Big]\cdot\frac{1}{(2^*)'}}\Big(\int_{\Omega} |v|^{2^*} \Big)^{\frac{\theta}{2^*}} \\
&\leqslant  C\|v\|^{\theta}_{W^{1,2}_0} \mbox{meas}(A^v_k)^{\Big[\frac{2^* -\theta (2^*)'}{2^*}\Big]\cdot\frac{1}{(2^*)'}}.
\end{align*}
Nonetheless, recall that
\begin{align*}
\int_\Omega |G_k(v)| = \int_{A^v_k} |G_k(v)| \leqslant{\rm \mbox{meas}}(A^v_k)^{\frac{1}{(2^*)'}} \left(\int_\Omega|G_k(v)|^{2^*}\right)^\frac{1}{2^*}.
\end{align*}
Thence,  the combination between the latter inequalities implies
\begin{align*}
\int_\Omega |G_k(v)| = \int_{A^v_k} |G_k(v)| \leqslant{\rm \mbox{meas}}(A^v_k)^{\alpha},
\end{align*}
where $\alpha=\frac{1}{(2^*)'}\Big[1+\frac{2^* - \theta(2^*)'}{2^*}\Big] > 1$, since $\frac{4}{N-2}> \theta$.
Therefore, once again by invoking Lemma $6.2$ in \cite{Boccardo}, we have $v\in L^\infty(\Omega)$
and
\begin{align*}
\|v\|_{L^\infty}\leqslant C\|\Phi\|_{L^m}.
\end{align*}
\end{proof}

With these tools at hand, we are able to prove existence of suitable solutions for a more favorable approximate version of our problem which will be explored in our investigation of \eqref{P}.

\begin{proposition}\label{cor1}
Let $\{f_k\}$ be a sequence of $L^\infty(\Omega)$ functions strongly convergent to $f$ in $L^m(\Omega),$ $m\geqslant1,$ for which
$
|f_k|\leqslant |f|
\mbox{ a.e. in } \Omega
$. Then, there exists  $(u_k,v_k) \in W^{1,2}_0(\Omega) \cap L^\infty(\Omega) \times W^{1,2}_0(\Omega)\cap L^\infty(\Omega)$, solution  to
\begin{align*}\tag{AP}\label{aprox}
\begin{cases}
\ -\mbox{div}(M(x)\nabla u_k)+g(x,u_k,v_k)= f_k \\ \ -\mbox{div}(M(x)\nabla v_k)=h(x,u_k,v_k)+\frac{1}{\tau _k},
\\
			\ \ \ \ \ u_k=v_k=0 \ \ \mbox{on} \ \ \partial \Omega,
\end{cases}
\end{align*}
where $\tau_k>0$ and $\tau_k \to + \infty$ if $k\to +\infty$. Moreover, if $f\ge0$ a.e. in $\Omega$ then $u_k\geqslant 0$  a.e. in $\Omega$ and $v_k > 0$ a.e. in $\Omega$.
\end{proposition}
\begin{proof}
Given $k>0$ consider $\tau> (c_2+d_2)C^{r+\theta}k^{r+\theta}$, where $C$ is given in Lemma \ref{lemma31} and $c_2$, $d_2$ in \eqref{P1}, \eqref{P3}. Let us recall the standard truncation $T_k(s)=\max (-k,\min(s,k))$ and then take $f_k=T_k(f)$. Thus, for $\Phi=f_k$, by combining Proposition \ref{prop1} and Lemma \ref{lemma31},  we obtain a couple $(u_k,v_k) \in W^{1,2}_0(\Omega) \cap L^\infty(\Omega) \times W^{1,2}_0(\Omega)\cap L^\infty(\Omega)$ solution for \eqref{P_A}.

Observe that $g_\tau(.,u_k,v_k)=g(.,u_k,v_k)$ and $h_\tau(.,u_k,v_k) = h(.,u_k,v_k)$ a.e. in $\Omega$. As a matter of fact, from \ref{P1}, Lemma \ref{lemma31} item $(ii)$ and the choice of $f_k$, we have
\begin{align*}
|g(x,u_k,v_k)| &\leqslant c_2 |u_k|^{r-1}|v_k|^{\theta+1}\\
&\leqslant c_2 C^{r+\theta}\|f_k\|_{L^\infty}^{r+\theta}\\
&\leqslant c_2 C^{r+\theta}k^{r+\theta}.
\end{align*}
Thence, since by the choice of $\tau>c_2 C^{r+\theta}k^{r+\theta}$, from \eqref{truncationshg} we have that $g_\tau$ coincides with $g$.  Analogously using the hypothesis \ref{P3} we conclude that  $h_\tau = h$.

Finally, remark that if $f\geqslant 0$ a.e. in $\Omega$ then $f_k \geqslant 0$ a.e. in $\Omega$ so that by the Weak Maximum Principle, $u_k\geq 0$ a.e. in $\Omega$.

Now we prove that $v_k>0$ a.e. in $\Omega$. In fact, consider $w_k \in C^{1,\alpha}(\Omega)$, with $0\leqslant \alpha <1$, the solution of
\begin{align}\label{2116}
    \begin{cases}
        \ -\mbox{div}(M(x)\nabla w_k) = \frac{1}{\tau_k} \  \ \mbox{in} \ \Omega, \\
        \ \ \ \ \ w_k = 0 \ \ \ \mbox{on} \ \ \partial \Omega.
    \end{cases}
\end{align}
Remark that, since $M\in W^{1,\infty}(\Omega)$  the existence of this $w_k$ is standard, for instance  see Corollary 8.36 in \cite{trudinger}. Then,  by a straightforward application of the Strong Maximum Principle of Vasquez in \eqref{2116}, see Theorem 4 in \cite{vazquez}, we obtain that $w_k>0$ in $\Omega$. After that, let us stress that
\begin{align*}
    - \mbox{div}(M(x)\nabla v_k) = h(x,u_k,v_k) + \frac{1}{\tau_k} \geqslant - \mbox{div}(M(x)\nabla w_k)
\end{align*}
then by the Comparison Principle  $v_k \geqslant w_k$ a.e. in $\Omega$ so that $v_k>0$ a.e. in $\Omega$.
\end{proof}

At this point, we are finally ready to obtain a set of delicate uniform a priori estimates which play a key role in our results. First, we  address the case with finite energy \[m\geqslant (r+\theta+1)^\prime.\]
Let us stress that we strongly use the fact that one of our approximate solutions of the decoupled system is strictly positive, i.e., $v_k$ satisfies $v_k>0$ a.e. in $\Omega$.
\begin{lemma}\label{teo47}
Let $f \in L^m(\Omega)$ where $m\geqslant (r+\theta+1)^\prime,$ and $f\geqslant 0$ a.e in $\Omega$, $r>1$ and $0<\theta<1.$ Then
\begin{align}\label{esti1}
\|u_k\|^2_{W^{1,2}_0} + \|v_k\|^2_{W^{1,2}_0} + \int_\Omega u_k^{r+\theta+1} + \int_{\Omega} u_k^{r} v_k^{\theta+1} \leqslant C \bigg(\|f\|_{L^m}^{\frac{r+\theta+1}{r+\theta}}+\frac{1}{\tau_k^2}\bigg),
\end{align}
where $C>0$ and $\tau_k\to\infty $ if $k\to \infty$.
\end{lemma}
\begin{proof}

Let us take $\varepsilon>0$. By considering $\psi =u_k^{\theta+1}(v_k+\varepsilon)^{-\theta}$ as a test function in the second equation of \eqref{aprox}, after dropping the positive term, we get
\begin{align*}
\int_\Omega h(x,u_k,v_k)u_k^{\theta+1}(v_k+\varepsilon)^{-\theta} &\leqslant (\theta+1) \int_\Omega M(x)\nabla v_k \cdot\nabla u_k u_k^\theta (v_k+\varepsilon)^{-\theta} \\
&-\theta\int_\Omega M(x)\nabla v_k \cdot \nabla v_k u_k^{\theta+1}(v_k+\varepsilon)^{-(\theta+1)}.
\end{align*}
Then by \eqref{P3}, \eqref{P4} and \eqref{P5}, it is clear that
%\begin{align}\nonumber
%\int_\Omega u_k^{r+\theta+1}v_k^{\theta+1}(v_k+\varepsilon)^{-(\theta+1)} &+ \int_\Omega M(x)\nabla v_k\cdot \nabla v_k u_k^{\theta+1} (v_k+\varepsilon)^{-(\theta+2)}\left[\theta v_k-\varepsilon\right]\\ \label{contaI}
%& \ \leqslant (\theta+1)\beta\int_\Omega|\nabla v_k||\nabla u_k| u_k^{\theta} (v_k+\varepsilon)^{-(\theta+1)}v_k,
%\end{align}

%and then,
\begin{align}
\label{contaI}
\nonumber\int_\Omega u_k^{r+\theta+1}v_k^{\theta}(v_k+\varepsilon)^{-\theta} &+ \alpha\theta\int_\Omega |\nabla v_k |^2 u_k ^{\theta+1} (v_k+\varepsilon)^{-(\theta+1)} \\ &\leqslant	(\theta+1)\beta\int_\Omega |\nabla v_k| |\nabla u_k| u_k^\theta (v_k+\varepsilon)^{-\theta}.
\end{align}
Before letting $\epsilon \to 0$, let us consider $\Omega_1=\{x\in\Omega;\ \frac{u_k}{v_k+\varepsilon}\leqslant 1\}$ and $\Omega_2=\{x\in\Omega;\ \frac{u_k}{v_k+\varepsilon}>1\}$, so that $\Omega=\Omega_1\cup \Omega_2$.
%\begin{align*}
%&\int_\Omega|\nabla v_k||\nabla u_k| u_k^\theta (v_k+\varepsilon)^{-\theta}=\int_{\Omega_1} |\nabla v_k||\nabla u_k| u_k^\theta (v_k+\varepsilon)^{-\theta} + \int_{\Omega_2} |\nabla v_k||\nabla u_k| u_k^\theta (v_k+\varepsilon)^{-\theta}.
%\end{align*}
In this way, by Young's inequality
\begin{align*}
(\theta+1)\beta\int_{\Omega} |\nabla v_k| |\nabla u_k| u_k^\theta (v_k+\varepsilon)^{-\theta}
 & \leqslant C_\eta (\theta+1)\beta \int_{\Omega} |\nabla u_k|^2 + \eta(\theta+1)\beta\int_{\Omega} |\nabla v_k|^2\frac{u_k^{2\theta}}{(v_k+\varepsilon)^{2\theta}}\\
\end{align*}
Now, remark that since $\frac{u_k}{v_k+\varepsilon}> 1$ in $\Omega_2$ and $2\theta<\theta+1$ , there follows that
\begin{align*}%\label{conta3}
\nonumber(\theta+1)\beta\int_\Omega |\nabla v_k| |\nabla u_k| u_k^\theta (v_k+\varepsilon)^{-\theta}
&\leqslant C_\eta(\theta+1)\beta\int_{\Omega} |\nabla u_k|^2 + \eta(\theta+1)\beta\int_{\Omega_1}|\nabla v_k|^2 \\
&+\eta(\theta+1)\beta\int_{\Omega_2} |\nabla v_k|^2 \frac{u_k^{\theta+1}}{(v_k+\varepsilon)^{\theta+1}}
\end{align*}
and hence, by combining the above estimate with \eqref{contaI}, we have
\begin{align*}
\int_\Omega u_k^{r+\theta+1}v_k^{\theta}(v_k+\varepsilon)^{-\theta} &+ \alpha\theta\int_\Omega |\nabla v_k |^2 u_k ^{\theta+1} (v_k+\varepsilon)^{-(\theta+1)} \leqslant C_\eta(\theta+1)\beta\int_{\Omega} |\nabla u_k|^2 \\
&+ \eta(\theta+1)\beta\int_{\Omega}|\nabla v_k|^2
+\eta(\theta+1)\beta\int_{\Omega} |\nabla v_k|^2 \frac{u_k^{\theta+1}}{(v_k+\varepsilon)^{\theta+1}}.
\end{align*}
Thus, by taking $\eta=\frac{\theta \alpha}{(\theta+1)\beta}$, and $C=\max\left\{C_\eta(\theta+1)\beta,\theta \alpha\right\}$ and by the Fatou Lemma, we arrive at
\begin{align}\label{eqth1.7g}
\int_\Omega u_k^{r+\theta+1} \leqslant \liminf_{\varepsilon \to 0} \int_\Omega u_k^{r+\theta+1}v_k^{\theta}(v_k+\varepsilon)^{-\theta} \leqslant C \left[\int_\Omega |\nabla v_k|^2 + \int_\Omega |\nabla u_k|^2 \right].
\end{align}
Now, let us proceed to the other estimates. Indeed, by choosing $\varphi=u_k$ in the first equation of \eqref{aprox}, it is clear that
\begin{align}\label{eqth1.8g}
	\alpha\int_\Omega |\nabla u_k|^2 + \int_\Omega g(x,u_k,v_k)u_k \leqslant \int_\Omega f_ku_k.
\end{align}
In particular, by combining H\"older's inequality with
\begin{align*}
c_1\int_\Omega u_k^r v_k^{\theta+1} \leqslant \int_\Omega g(x,u_k,v_k)u_k \leqslant \int_\Omega f u_k.
\end{align*}
%then
%\begin{align*}
%c_1 \int_\Omega u_k^r v_k^{\theta+1}\leqslant \|f\|_{L^m} %\|u_k\|_{L^{m^\prime}},
%\end{align*}
then
\begin{align}\label{eqth1.9g}
	\int_\Omega u_k^rv_k^{\theta+1} \leqslant C\|f\|_{L^m}\|u_k\|_{L^{r+\theta+1}},
\end{align}
where we strongly used that $f\in L^m(\Omega)$ for $m\geq (r+\theta+1)^\prime$.

Further, by taking $\psi=v_k$ in the second equation of \eqref{aprox}, from\eqref{P3} and \eqref{P4},  it is clear that
\begin{align*}
	\alpha\int_\Omega |\nabla v_k|^2 \leqslant \int_\Omega \Big(h(x,u_k,v_k)+\frac{1}{\tau_k}\Big)v_k &\leqslant d_2\int_\Omega u_k^r v_k^{\theta+1}+\frac{1}{\tau_k}\int_\Omega v_k \\
 &\leqslant C\|f\|_{L^m} \|u_k\|_{L^{r+\theta+1}}+\frac{C}{\tau_k}\|\nabla v_k\|_{L^2}.
 \end{align*}
However, it is clear that
\[\frac{C}{\tau_k}\|\nabla v_k\|_{L^2} \leq \frac{C_\alpha}{\tau_k^2}+\frac{\alpha}{2}\|\nabla v_k\|^2_{L^2}\]
whereas, by combining with \eqref{eqth1.8g}, give us
\begin{align}\label{grad}
	\int_\Omega |\nabla u_k|^2 + \int_\Omega |\nabla v_k|^2 \leqslant C \big(\|f\|_{L^m} \|u_k\|_{L^{r+\theta+1}}+\frac{1}{\tau_k^2}\big).
\end{align}
In particular, by \eqref{eqth1.7g} and Young's inequality, we obtain
\begin{align}\label{estap}
	\|u_k\|^{r+\theta+1}_{L^{r+\theta+1}} \leqslant C \bigg(\|f\|_{L^m}^\frac{r+\theta+1}{r+\theta}+\frac{1}{\tau_k^2}\bigg).
\end{align}
Therefore, gathering \eqref{eqth1.9g}, \eqref{estap} with \eqref{grad}, we finally have
\begin{align*}
\|u_k\|^2_{W^{1,2}_0}+\|v_k\|^2_{W^{1,2}_0}+\int_{\Omega} u_k^{r} v_k^{\theta+1} \leqslant C \bigg(\|f\|_{L^m}^{\frac{r+\theta+1}{r+\theta}}+\frac{1}{\tau_k^2}\bigg),
\end{align*}
where we stress that $\tau_k\to \infty$ if $k\to\infty$
\end{proof}

Next, we will address the case   \[\left(\frac{r-\theta+1}{1-2\theta}\right)^\prime<m<(r+\theta+1)^\prime,\]
which involves even more delicate estimates. Indeed, observe that since $m< (r+\theta+1)^\prime$ the method used in Lemma \ref{teo47} needs to the modified in a way that comprises estimates under the energy regimes. In order to do that, we have to consider another set of test functions to compensate the additional singularity in our system.

\begin{lemma}\label{lem4.10}
Let $f \in L^m(\Omega)$ with $f\geqslant 0$, $\left(\frac{r-\theta+1}{1-2\theta}\right)'\leqslant m<(r+\theta+1)'$ and $0<\theta<\frac{1}{2}$. Then
\begin{align}\label{esti2}
&\nonumber\|u_k\|^p_{W^{1,p}_0}+\int_\Omega u_k^{r-\theta+1} + \int_\Omega u_k^r v_k^{1-\theta} \leqslant C\bigg(\|f\|_{L^m}^{\frac{r-\theta+1}{r+\theta}}+\frac{1}{\tau_k}\bigg)  \ \mbox{ and }\\
&\|v_k\|_{W^{1,q}_0}^q \leqslant C\bigg( \|f\|_{L^m}^{\frac{r-\theta+1}{r+\theta}}+\frac{1}{\tau_k}\bigg)^\frac{N}{N-2\theta}
\end{align}
where $p=\frac{2(r-\theta+1)}{r+\theta+1}$, $q=\frac{2N(1-\theta)}{N-2\theta}$ and  $C>0$, where $\tau_k\to\infty $ if $k\to \infty$.  Moreover, if $r\geqslant\frac{N+2}{N-2}$, then $\{u_k\}$ and $\{v_k\}$ are bounded in $W^{1,q}_0(\Omega)$.
\end{lemma}
\begin{proof}
Let us consider $\varphi_\varepsilon=(u_k+\varepsilon)^\gamma-\varepsilon^\gamma$, $0<\epsilon \leq 1$, $0<\gamma<1$, as a test function in the first equation of \eqref{aprox}, so that
\begin{align}\label{varepi}
\alpha \gamma \int_\Omega |\nabla u_k|^2(u_k+\varepsilon)^{\gamma-1} + \int_\Omega g(x,u_k,v_k)\left[(u_k+\varepsilon)^\gamma-\varepsilon^\gamma\right] \leqslant \int_\Omega f_k \left[(u_k+\varepsilon)^\gamma-\varepsilon^\gamma\right].	
\end{align}
Then, by using that $u_k\in L^\infty(\Omega)$ and the Dominated Convergence Theorem, it is clear that
%For the right hand side, if $0<\varepsilon \leqslant 1,$ we have
%\begin{align*}
%|f_k[(u_k+\varepsilon)^\gamma-\varepsilon^\gamma]| \leqslant |f|\cdot[(\|u_k\|_{L^\infty}+\varepsilon)^\gamma+\varepsilon^\gamma].
%\end{align*}
%By taking $0<\varepsilon\leqslant1$, we have
%%\begin{align*}
%|f_k[(u_k+\varepsilon)^\gamma-\varepsilon^\gamma]| \leqslant %C|f|\cdot(\|u_k\|_{L^\infty}^\gamma+1) \ \in \ L^1(\Omega).
%\end{align*}
%Moreover
%\begin{align*}
%f_k[(u_k+\varepsilon)^\gamma-\varepsilon^\gamma]\to f_ku_k^\gamma \ \ %%\mbox{a.e. in} \ \ \Omega \ \ \mbox{when} \ \ \varepsilon\to0^+,
%\end{align*}
%and then, by the Dominated Convergence Theorem, one has that
\begin{align*}
\lim_{\varepsilon\to0}\int_\Omega f_k[(u_k+\varepsilon)^\gamma-\varepsilon^\gamma]=\int_\Omega f_ku_k^\gamma.
\end{align*}
Thus, by taking $\varepsilon\to0$ in \eqref{varepi}, by recalling  Lemma \ref{teclema} and  \eqref{P1}, we end up with
%\begin{align*}%\label{tec1}
%\alpha\gamma\int_\Omega |\nabla u_k|^2 u_k^{\gamma-1} + \int_\Omega %h(x,u_k,v_k) u_k^\gamma  \leqslant \int_\Omega f u_k^\gamma,
%	\end{align*}
%%%%%%%%%%%%%+
%so that by we end up with
\begin{align}\label{uhgam}
\alpha\gamma \int_{\Omega_+^k} |\nabla u_k|^2 u_k^{\gamma-1} + c_1 \int_\Omega u_k^{r-1+\gamma}v_k^{\theta+1} \leqslant \int_\Omega f u_k^\gamma<\infty,
\end{align}
which is finite for every fixed $k\in \mathbb{N}$, where $\Omega_+^k=\{u_k>0\}$.
Now, we address the set of estimates arising as a byproduct of {\bf the coupling }between both equations of our system.
We must stress that for us, the fact that $v_k>0$ a.e. in $\Omega$ will be crucial. As a matter of fact, by considering $\psi = (v_k+\varepsilon)^{\gamma} - \varepsilon^\gamma$ in the second equation of \eqref{aprox} it is clear that
	\begin{align}\label{numerou}
			\alpha \gamma \int_\Omega |\nabla v_k|^2(v_k+\varepsilon)^{\gamma-1} \leqslant \int_\Omega h(x,u_k,v_k)[(v_k+\varepsilon)^{\gamma}-\varepsilon^\gamma]+\frac{1}{\tau_k}\int_\Omega (v_k+\varepsilon)^\gamma-\varepsilon^\gamma.	
	\end{align}
Then, once more, by taking $\varepsilon\to0$ and by recalling \eqref{P3}, from Lemma \ref{teclema}, we obtain
\begin{align}\label{vggam}  \alpha \gamma \int_\Omega |\nabla v_k|^2v_k^{\gamma-1} \leqslant \int_\Omega \Big(h(x,u_k,v_k)+\frac{1}{\tau_k}\Big)v_k^{\gamma}\leqslant d_2 \int_\Omega u_k^r v_k^{\gamma+\theta}+\frac{1}{\tau_k}\int_\Omega v_k^\gamma %< \infty,
\end{align}
%which is finite for every $k$ fixed.

Further,  take $\psi=(u_k+\varepsilon)^{\gamma+\theta}(v_k+\varepsilon)^{-\theta}$ in the second equation of \eqref{aprox}. By dropping the positive term, one has that
	\begin{align*}
		 \int_\Omega h(x,u_k,v_k) (u_k+\varepsilon)^{\gamma+\theta} (v_k+\varepsilon)^{-\theta}&\leqslant  (\gamma+\theta)\int_\Omega M(x)\nabla v_k\cdot \nabla u_k\ (u_k+\varepsilon)^{\gamma+\theta-1}(v_k+\varepsilon)^{-\theta}	\\
  &-\theta\int_\Omega M(x)\nabla v_k \cdot\nabla v_k (u_k+\varepsilon)^{\gamma+\theta}(v_k+\varepsilon)^{-(\theta+1)}
	\end{align*}
	Then, by \eqref{P3} and \eqref{P4} it is clear that
	\begin{align}\label{tec3}
		\nonumber
  c_1\int_\Omega u_k^r v_k^{\theta}(u_k+\varepsilon)^{\gamma+\theta}  (v_k+\varepsilon)^{-\theta}
  &+\alpha\theta \int_\Omega |\nabla v_k|^2 (u_k+\varepsilon)^{\gamma+\theta} (v_k+\varepsilon)^{-(\theta+1)}
		 \\
&\nonumber\leqslant  \beta(\gamma+\theta)\int_\Omega |\nabla v_k||\nabla u_k|(u_k+\varepsilon)^{\gamma+\theta-1} (v_k+\varepsilon)^{-\theta}
\\
&= \beta(\gamma+\theta)\int_{\Omega_+^k} |\nabla v_k||\nabla u_k|(u_k+\varepsilon)^{\gamma+\theta-1} (v_k+\varepsilon)^{-\theta},
	\end{align}
since $\nabla u_k =0 $ a.e. in the set $u_k=0$, where $\Omega_+^k=\{u_k>0\}$.

Given $\eta >0$, by Young's inequality, we have that
	\begin{align*}
\int_{\Omega^k_+} |\nabla v_k||\nabla u_k| &(u_k+\varepsilon)^{\gamma+\theta-1} (v_k+\varepsilon)^{-\theta}\leqslant \int_{\Omega^k_+} |\nabla v_k||\nabla u_k| (u_k+\varepsilon)^{\gamma+\theta-1} |v_k|^{-\theta}\\
&\leqslant \eta\int_{\Omega} |\nabla v_k|^2 v_k^{-2\theta} + C_\eta\int_{\Omega^k_+} |\nabla u_k|^2 (u_k+\varepsilon)^{2(\gamma+\theta-1)}.	
	\end{align*}
By combining the above inequality and \eqref{tec3} we get
\begin{align}
\nonumber
  c_1\int_\Omega u_k^r v_k^{\theta}(u_k&+\varepsilon)^{\gamma+\theta}  (v_k+\varepsilon)^{-\theta}
  +\alpha\theta \int_\Omega |\nabla v_k|^2 (u_k+\varepsilon)^{\gamma+\theta} (v_k+\varepsilon)^{-(\theta+1)}\\
  \nonumber
&\leqslant \beta(\gamma+\theta)\eta\int_\Omega |\nabla v_k|^2 v_k^{-2\theta} + \beta(\gamma+\theta)C_\eta\int_{\Omega^k_+} |\nabla u_k|^2 (u_k+\varepsilon)^{2(\gamma+\theta-1)}.
\end{align}
 At this point, it is natural to choose an adequate $\gamma$ in order to guarantee that certain crucial exponents coincide, what allows us to explore the coupling between the equations. Indeed, by fixing $\gamma=1-2\theta$ so that $2(\gamma+\theta-1)=\gamma -1=-2\theta$, after dropping the positive term there follows that
\begin{align}\label{tec4}
\nonumber
  c_1\int_\Omega u_k^r v_k^{\theta}(u_k+\varepsilon)^{1-\theta}  (v_k+\varepsilon)^{-\theta}
 &\leqslant \beta(1-\theta)\eta\int_\Omega |\nabla v_k|^2 v_k^{-2\theta}
 \\&+ \beta(1-\theta)C_\eta\int_{\Omega_+^k} |\nabla u_k|^2 (u_k+\varepsilon)^{-2\theta}.
\end{align}

%	So
%	\begin{align}\label{tec4}
%		\nonumber &\frac{(\gamma+\theta)}{2}\int_\Omega |\nabla v_k|^2 v_k^{-2\theta} + \frac{(\gamma+\theta)}{2}\int_\Omega |\nabla u_k|^2 u_k^{2(\gamma+\theta+1)} \geqslant\\
%		& \ \ -\int_\Omega |\nabla v_k|^2 u_k^{\gamma+\theta} (v_k^2+\varepsilon)^{-\frac{(\theta+3)}{2}}\varepsilon + \int_\Omega |u_k|^{r+\gamma+\theta} |v_k|^{\theta+1} (v_k^2+\varepsilon)^{-\frac{(\theta+1)}{2}}.
%	\end{align}
However, remark that by \eqref{uhgam}, for every $k$ fixed, we have $|\nabla u_k|^2u_k^{-2\theta} \in L^1(\Omega_+^k)$. Thus by taking $\varepsilon\to0$ in \eqref{tec4}, employing the Fatou Lemma combined with the Dominated Convergence Theorem, we arrive at
%	In addition, we have that
%	\begin{align*}
%		-\liminf_{\varepsilon\to 0}\int_\Omega|\nabla v_k|^2 \frac{u_k^{\gamma+\theta}}{(v_k^2+\varepsilon)^{\frac{\theta+3}{2}}} \geqslant -\limsup_{\varepsilon\to0}\int_\Omega |\nabla v_k|^2 \frac{u_k^{\gamma+\theta}}{|v_k|^{1+\theta}}\varepsilon = 0.
%	\end{align*}
%	Thus, applying Fatou Lemma in \eqref{tec4}, we have that
	\begin{align}%\label{tec5}
 \label{tec6}
	 c_1\int_\Omega u_k^{r-\theta+1} \leqslant\beta(1-\theta)\eta\int_\Omega |\nabla v_k|^2 v_k^{-2\theta} + \beta(1-\theta)C_\eta\int_{\Omega_+^k} |\nabla u_k|^2 u_k^{-2\theta}.
	\end{align}
	
 %At this point, it is natural to choose an adequate $\gamma$ in order to guarantee that certain crucial exponents coincide, what allows us to explore the coupling between the equations. Indeed, by fixing $\gamma=1-2\theta$ so that $2(\gamma+\theta-1)=\gamma -1=-2\theta$, there follows that
%	\begin{align}
 %c_1\int_\Omega u_k^{r-\theta+1} \leqslant\beta(1-\theta)\eta\int_\Omega |\nabla v_k|^2 v_k ^{-2\theta} + \beta(1-\theta)C_\eta\int_\Omega |\nabla u_k|^2u_k^{-2\theta}.	
	%\end{align}
	Now, it is clear that
	%\begin{align}\label{tec7}
	%\int_\Omega |u|^r |v|^{\theta+\gamma} = \int_{\{|u|\geqslant|v|\}} |u|^r |v|^{\theta+\gamma} + \int_{\{|u|\leqslant|v|\}} |u|^r |v|^{\theta+\gamma}.	
	%\end{align}
	\begin{align*}
	   \int_{\{u_k\leqslant v_k\}} u_k^r v_k ^{1-\theta}\leqslant \int_{\{u_k \leqslant v_k\}} u_k^r u_k^{-2\theta}v_k^{\theta+1}
    %&=\int_{\{u_k \leqslant v_k\}}u_k^{r-2\theta} v_k^{\theta+1} \\
    \leqslant \int_{\Omega}u_k^{r-2\theta}v_k^{\theta+1} \mbox{ and }
    \int_{\{u_k\geqslant v_k\}} u_k^r v_k^{1-\theta}\leqslant  \int_\Omega u_k^{r-\theta+1}.
	\end{align*}
	%On the other hand,
	%\begin{align*}
%		\int_{\{u_k\geqslant v_k\}} u_k^rv_k^{\gamma+\theta}=
%\int_{\{u_k\geqslant v_k\}} u_k^r v_k^{1-\theta}\leqslant \int_{\{u_k\geqslant v_k\}} u_k^r u_k^{1-\theta} \leqslant \int_\Omega u_k^{r-\theta+1}.
%	\end{align*}
%	Since $\gamma+\theta-1=-\theta \iff \gamma+\theta = 1-\theta.$\\ \\
	The latter estimates clearly guarantee that
	\begin{align}\label{tetagama}
		\int_\Omega u_k ^r v_k^{1-\theta} \leqslant \int_\Omega u_k^{r-\theta+1} + \int_\Omega u_k^{r-2\theta}v_k^{\theta+1}.
	\end{align}
Thus, by gathering \eqref{vggam}, \eqref{tec6} and \eqref{tetagama}, and by using \eqref{uhgam} twice, since $\gamma+\theta=1-\theta$, there follows
\begin{align*}%\label{tec8}
\nonumber
\dfrac{\alpha(1-2\theta)}{d_2}\int_\Omega |\nabla v_k|^2 v_k^{-2\theta}&\leqslant \int_\Omega u_k^{r-\theta+1} + \int_\Omega u_k^{r-1+\gamma}v_k^{\theta+1} +\frac{1}{\tau_k}\int_\Omega v_k^{1-2\theta}
\\
 &\leqslant \dfrac{\beta(1-\theta)\eta}{c_1}\int_\Omega |\nabla v_k|^2v_k^{-2\theta} + C \int_\Omega f u_k^{1-2\theta}
 +\frac{1}{\tau_k}\int_\Omega v_k^{1-2\theta},	
\end{align*}
where we combined \eqref{uhgam} and \eqref{tec6} in the right-hand side. If we consider $\eta=\frac{\alpha(1-2\theta)c_1}{2\beta(1-\theta)d_2}$, it is clear that
%Since $0<\theta<\frac{1}{3}$ we obtain that $C_0=\gamma-\frac{(1-\theta)}{2}$ is a strictly positive constant. So
\begin{align}\label{tec91}
\int_\Omega|\nabla v_k|^2 v_k^{-2\theta} \leqslant C\int_\Omega f u_k^{1-2\theta}+\frac{1}{\tau_k}\int_\Omega v_k^{1-2\theta}< \infty,
\end{align}
which is finite for every $k$ fixed.

At this point, for the sake of simplicity, observe that, given $\varepsilon>0$, by the locally Lipschitz Chain Rule,
$|\nabla v_k|^2(v_k+\varepsilon)^{-2\theta} = \frac{1}{(1-\theta)^2} |\nabla (v_k+\varepsilon)^{1-\theta}|^2$. Thence, by combining the Sobolev Embedding, the Fatou Lemma and the Dominated Convergence Theorem, we have
\begin{align}
\label{2046}
\nonumber
 \left(\int_\Omega v_k^{(1-\theta)2^*}\right)^{\frac{2}{2^*}}&\leqslant \liminf_{\varepsilon\to0^+} \left(\int_\Omega (v_k+\varepsilon)^{(1-\theta)2^*}\right)^{\frac{2}{2^*}}
 \\\nonumber
 &\leqslant
 \liminf_{\varepsilon\to0^+} C\left(\int_\Omega |\nabla (v_k+\varepsilon)^{1-\theta}|^2\right)^{\frac{2}{2^*}}
 \\
 &=\int_\Omega|\nabla v_k|^2 v_k^{-2\theta},
\end{align}
 where in the latter estimates we used that $|\nabla v_k|^2 v_k^{-2\theta}\in L^1(\Omega)$ for every $k\in\mathbb{N}$, fixed. Moreover, observe that $\frac{(1-\theta)2^*}{1-2\theta}>1$ and thus, by combining  H\"older's inequality for $\frac{(1-\theta)2^*}{1-2\theta}$ and $\frac{(1-\theta)2N}{N+2-4\theta}$, with the last estimate, we arrive at
\begin{align}
\label{sobv}
\nonumber
\int_\Omega v_k^{1-2\theta} &\leq C \bigg(\int_\Omega v_k^{(1-\theta)2^*}\bigg)^\frac{1-2\theta}{(1-\theta)2^*}
\\
&\leq C \bigg(\int_\Omega |\nabla v_k|^2 v_k^{-2\theta}\bigg)^\frac{1-2\theta}{2-2\theta}.
 \end{align}
 %C\|f\|_{L^m}^{\frac{r-\thea+1}{r+\theta}}.
%\end{align}
Further, from \eqref{tec91} and \eqref{sobv} with Young's inequality for $\frac{2-2\theta}{1-2\theta}$ and $2-2\theta$, we obtain
\begin{align}\nonumber
\int_\Omega|\nabla v_k|^2 v_k^{-2\theta} &\leqslant C\int_\Omega f u_k^{1-2\theta}+\frac{1}{\tau_k}C \bigg(\int_\Omega |\nabla v_k|^2 v_k^{-2\theta}\bigg)^\frac{1-2\theta}{2-2\theta}
\\\nonumber
&\leqslant C \bigg(\int_\Omega f u_k^{1-2\theta}+\frac{1}{\tau_k}\bigg)+\frac{1}{2\tau_k}\int_\Omega |\nabla v_k|^2 v_k^{-2\theta}
\end{align}
which clearly guarantees that
\begin{align}\label{tec9}
\int_\Omega|\nabla v_k|^2 v_k^{-2\theta} &\leqslant C\bigg(\int_\Omega f u_k^{1-2\theta}+\frac{1}{\tau_k}\bigg).\end{align}

Thus, by \eqref{tec6}, \eqref{uhgam} and \eqref{tec9}, we get
\begin{align*}
\int_\Omega u_k^{r-\theta+1} &\leqslant C\Big(\int_\Omega f u_k^{1-2\theta}+\frac{1}{\tau_k}\Big)
\\
&\leqslant C\Big(\|f\|_{L^m}||u_k||^{1-2\theta}_{L^{(1-2\theta)m^{\prime}}}+\frac{1}{\tau_k}\Big)\\
&\leqslant C\Big(\|f\|_{L^m}||u_k||^{1-2\theta}_{L^{r-\theta+1}}+\frac{1}{\tau_k}\Big),
\end{align*}
where we used that $(1-2\theta)m^\prime=\gamma m^\prime \leqslant r-\theta+1$, for $m^\prime\leqslant\frac{r-\theta+1}{1-2\theta}$.
Further, by means of another application of the Young inequality, for $\frac{r-\theta+1}{1-2\theta}$ and $\frac{r-\theta+1}{r+\theta}$, after straightforward compensations, we end up with
\begin{align*}
\int_\Omega u_k^{r-\theta+1}\leqslant C \bigg(\|f\|^{\frac{r-\theta+1}{r+\theta}}_{L^m}+\frac{1}{\tau_k}\bigg).	
\end{align*}
In particular, by \eqref{tec9},	
\begin{align}\label{tec10}
\int_\Omega u_k^{r-\theta+1} \leqslant C\bigg(\|f\|^{\frac{r-\theta+1}{r+\theta}}_{L^m}+\frac{1}{\tau_k}\bigg) \mbox{ and } \int_\Omega|\nabla v_k|^2 v_k^{-2\theta} \leqslant C \bigg(\|f\|^{\frac{r-\theta+1}{r+\theta}}_{L^m}+\frac{1}{\tau_k}\bigg).
\end{align}
Further, by gathering \eqref{uhgam}, \eqref{tetagama}, with an analogous argument used to prove \eqref{tec10} we obtain
\[\int_\Omega u_k ^r v_k^{1-\theta}  \leq C \bigg(\|f\|_{L^m}^{\frac{r-\theta+1}{r+\theta}}+\frac{1}{\tau_k}\bigg).\]

As a final step, we will handle the coupling estimates for the gradients. Indeed, on one hand,  observe that for $1\leqslant p<2$, by H\"{o}lder's inequality with exponents $\frac{2}{p}$ and $\frac{2}{2-p}$, we get

\begin{align*}
\int_\Omega |\nabla u_k|^p = \int_\Omega \frac{|\nabla u_k|^p}{(u_k+\varepsilon)^{\theta p}}(u_k+\varepsilon)^{\theta p}&\leqslant \left(\int_\Omega|\nabla u_k|^2 (u_k+\varepsilon)^{-2\theta}\right)^{\frac{p}{2}} \cdot \left( \int_\Omega (u_k+\varepsilon)^{\frac{2\theta p}{2-p}}\right)^{\frac{2-p}{2}} \\
&=\left(\int_{\Omega^k_+}|\nabla u_k|^2 (u_k+\varepsilon)^{-2\theta}\right)^{\frac{p}{2}} \cdot \left( \int_\Omega (u_k+\varepsilon)^{\frac{2\theta p}{2-p}}\right)^{\frac{2-p}{2}}.
\end{align*}
Then, since by \eqref{uhgam}, $|\nabla u_k|^2u_k^{-2\theta} \in L^1(\Omega^k_+)$, from the Lebesgue Convergence Theorem,  we can take $\varepsilon \to 0$ so that
\begin{align*}
\int_\Omega |\nabla u_k|^p \leqslant \left(\int_{\Omega^k_+}|\nabla u_k|^2 u_k^{-2\theta}\right)^{\frac{p}{2}} \cdot \left( \int_\Omega u_k^{\frac{2\theta p}{2-p}}\right)^{\frac{2-p}{2}}.
\end{align*}
By choosing $\frac{2\theta p}{2-p}=r-\theta+1$, i.e., $p=\frac{2(r-\theta+1)}{r+\theta+1}$, from \eqref{uhgam} and \eqref{tec10}, there follows that
\begin{align*}
\int_\Omega |\nabla u_k|^p \leqslant C \bigg(\|f\|_{L^m}^{\frac{r-\theta+1}{r+\theta}}+\frac{1}{\tau_k}\bigg).
\end{align*}
%where $C=C(\|f\|_{L^m})>0$.
%Since
%\begin{align*}
%\left[(1-\gamma)\frac{q}{2}\right]\frac{2}{2-q} = \frac{(1-\gamma)q}{2-q} = r+1-\theta,
%\end{align*}
%i.e.
%\begin{align*}
%q=\frac{2(r+1-\theta)}{r+\theta+1},
%\end{align*}
%follows that
%\begin{align*}
%\int_\Omega |\nabla u_k|^q &\leqslant C \|f\|_{L^m}^{(1+m)\frac{q}{2}} \|f\|_{L^m}^{\frac{2m}{2-q}}\\
%&\leqslant C \|f\|_{L^m}^{\widetilde{m}},
%\end{align*}
%where $\widetilde{m}=(1+m)\frac{q}{2}+\frac{2m}{2-q}.$\\ \\
Finally, by combining \eqref{2046} and \eqref{tec10}, we already know that

%Now, on the other hand, observe that  for $\epsilon>0$, from \eqref{tec10} one clearly has that
%\begin{align*}%\label{tec11}
%\int_\Omega |\nabla v_k|^2 (v_k+\varepsilon)^{-2\theta} \leqslant \int_\Omega |\nabla v_k|^2 v_k^{-2\theta} \leqslant C\|f\|_{L^m}^{\frac{r-\theta+1}{r+\theta}}.
%\end{align*}
%However, by the locally Lipschitz Chain Rule,
%$|\nabla v_k|^2(v_k+\varepsilon)^{-2\theta} = \frac{1}{(1-\theta)^2} |\nabla (v_k+\varepsilon)^{1-\theta}|^2$, so that from the Sobolev Embedding combined with the Fatou Lemma, we get

\begin{align}
\label{sobv1}
 \left(\int_\Omega v_k^{(1-\theta)2^*}\right)^{\frac{2}{2^*}} \leqslant
 C\bigg(\|f\|_{L^m}^{\frac{r-\theta+1}{r+\theta}}+\frac{1}{\tau_k}\bigg).
\end{align}
Moreover, recalling that $v_k>0$ a.e. in $\Omega$, and then for $1\leqslant q<2$ by  H\"{o}lder's inequality with exponent  $\frac{2}{q}$, we get
\begin{align*}
\int_\Omega |\nabla v_k|^{q} \leqslant \left(\int_\Omega|\nabla v_k|^2 v_k^{-2\theta}\right)^\frac{q}{2} \cdot \left(\int_\Omega v_k ^{\frac{2\theta q}{2-q}}\right)^{\frac{2-q}{2}}.
\end{align*}
Hence, so that it is enough to choose $(1-\theta)2^*=\frac{2\theta q}{2-q}$,i.e., $q=\frac{2N(1-\theta)}{N-2\theta}$, and  by \eqref{tec9} and \eqref{sobv1}, after straightforward computations, one arrives at
\begin{align*}
\int_\Omega |\nabla v_k|^{q} \leqslant C\bigg( \|f\|_{L^m}^{\frac{r-\theta+1}{r+\theta}}+\frac{1}{\tau_k}\bigg)^\frac{N}{N-2\theta}, \mbox{ where } C>0.
\end{align*}

In addition, remark that by the choice of $q$, if $r\geqslant\frac{N+2}{N-2}$, then $q<p$. Thus, by the gradient estimates obtained above, its clear that both $\{u_k\}$ and $\{v_k\}$ are bounded in $W^{1,q}_0(\Omega)$.
%\newpage

\end{proof}

\section{Proof of the Main Results}
We are now in position to prove our Main Results.
\subsection{Proof of Theorem \ref{lemah}}
By Lemma \ref{teo47}, there exist  $\{u_k\}$, $\{v_k\} \subset W^{1,2}_0(\Omega)$, and $u$, $v$ in $W^{1,2}_0(\Omega)$ such that, up to subsequences relabeled the same,
\begin{align*}
	&u_k\rightharpoonup u \ \mbox{weakly} \ \mbox{in} \ W^{1,2}_0(\Omega), \ u_k \to u \ \mbox{in} \ L^{s_1}(\Omega), \ \mbox{and a.e. in} \ \Omega;\\
	\nonumber&v_k\rightharpoonup v \ \mbox{weakly} \ \mbox{in} \ W^{1,2}_0(\Omega), \ v_k \to v \ \mbox{in} \ L^{s_2}(\Omega), \ \mbox{and a.e. in} \ \Omega,
\end{align*}
for $u\geqslant 0$, $v\geqslant 0$ a.e. in $\Omega$, and
\[\|u\|^2_{W^{1,2}_0} + \|v\|^2_{W^{1,2}_0} + \int_\Omega u^{r+\theta+1} + \int_{\Omega} u^r v^{\theta+1} \leqslant C \|f\|_{L^m}^{\frac{r+\theta+1}{r+\theta}},\]
where $s_1< \max\{2^*,r+\theta+1\}$ and $s_2< 2^*$, since $\tau_k \to +\infty$.

In order to prove that the couple satisfies \eqref{P_F} it is enough to show that
\begin{align*}
	g(x,u_k,v_k) \to g(x,u,v) \ \ \mbox{in} \ L^1(\Omega)	
\mbox{ and }
	h(x,u_k,v_k) \to h(x,u,v) \ \ \mbox{in} \ L^1(\Omega).	
\end{align*}
For this, take  $\lambda>0$, and consider $\varphi=\frac{T_\lambda(G_n(u_k))}{\lambda}$ as a test function in the first equation of \eqref{aprox} so that
\begin{align}\label{httf}
 \int_\Omega M(x) \nabla u_k \cdot \nabla u_k G'_n(u_k) \frac{T_\lambda'(G_n(u_k))}{\lambda} + \int_\Omega g(x,u_k,v_k) \frac{T_\lambda(G_n(u_k))}{\lambda} \leq \int_{\{u_k>n\}} f_k .
\end{align}
However, since $T_k$ and $G_n$ are both monotone, by using \eqref{P5} it is clear that
\[\int_\Omega M(x) \nabla u_k \cdot \nabla u_k G'_n(u_k) \frac{T_\lambda'(G_n(u_k))}{\lambda}\geqslant \alpha\int_\Omega |\nabla u_k|^2 G'_n(u_k) \frac{T_\lambda'(G_n(u_k))}{\lambda}\geqslant 0.\]
Further, by \eqref{P22} and by the very definition of $T_\lambda$ and $G_n$
\[\int_\Omega g(x,u_k,v_k) \frac{T_\lambda(G_n(u_k))}{\lambda}
\geqslant\int_{\{u_k>n+\lambda\}} g(x,u_k,v_k),
\]
Hence, by combining the latter estimates with \eqref{httf}, we get
\begin{align*}
\int_{\{u_k>n+\lambda\}} g(x,u_k,v_k)  \leqslant \int_{\{u_k >n\}} f,
\end{align*}
 and then, by letting $\lambda\to 0$, from \eqref{P1}, we have
\begin{align*}
\int_{\{ u_k >n\}} u_k^{r-1}v_k^{\theta+1} \leqslant C\int_{\{u_k>n\}} f.
\end{align*}
In this way, as a direct consequence of Lemma \ref{lem48}, for $s=t=2$ there follows that $ g(x,u_k,v_k) \to g(x,u,v) \ \ \mbox{in} \ L^1(\Omega)$, where clearly $2> \frac{N(\theta+1)}{N+\theta+1}$, for all $0<\theta<1$.

It remains to prove that
$
h(x,u_k,v_k) \to h(x,u,v) \ \mbox{in} \ L^1(\Omega).
$
First observe that we can consider up to subsequences relabeled the same, that $\frac{1}{\tau_k}<1$. Further, remark that by estimate \eqref{esti1} and H\"older's inequality,
\begin{align*}
    \int _E u_k^r &\leqslant \left(\int_E u_k^{r+\theta+1}\right)^\frac{r}{r+\theta+1}\cdot {\rm \mbox{meas}}(E)^{\frac{\theta+1}{r+\theta+1}}
    \\
    &\leqslant C \bigg(\|f\|^{\frac{r+\theta+1}{r+\theta}}_{L^m}+1\bigg)^{\frac{r}{r+\theta+1}}\bigg({\rm \mbox{meas}}(E)\bigg)^{\frac{\theta+1}{r+\theta+1}},
\end{align*}
and then ${u_k^r}$ is clearly uniformly integrable. Moreover, by \eqref{esti1} it is clear that \[\int_{\{v_k>n\}} u_k^r v_k^{\theta+1} \leq C \bigg(\|f\|_{L^m}^{\frac{r+\theta+1}{r+\theta}}+1\bigg),\]
clearly uniform with respect to $k$. Thence, once again by Lemma \ref{lem48} for once again $s=t=2$, we have
\begin{align*}
h(x,u_k,v_k) \to h(x,u,v) \ \mbox{in} \ L^1(\Omega).
\end{align*}
And also, it is trivial that
\[\int_\Omega \frac{1}{\tau_k}\psi\to 0.\]

Therefore, by letting $k\to +\infty$ in \eqref{aprox},  we arrive at
\begin{align*}
\int_{\Omega}\nabla u\cdot\nabla\varphi  + \int_{\Omega} g(x,u,v)\varphi = \int_{\Omega} f\varphi \ \ \ \ \forall\ \varphi \in C^\infty_c(\Omega)
\\
 \int_{\Omega}\nabla v\cdot\nabla\psi = \int_{\Omega} h(x,u,v)\psi \ \ \ \ \ \forall\ \psi \in C^\infty_c(\Omega),
\end{align*}
where $g(x,u,v)$, $h(x,u,v) \in L^1(\Omega)$ and the result follows.

%\newpage

\subsection{Proof of Theorem \ref{principal}}
This proof is analogous to the proof of Theorem \ref{lemah}, what allows us to omit certain details. In this fashion, from Lemma \ref{lem4.10},
there exist  $\{u_k\} \subset W^{1,p}_0(\Omega)$, $\{v_k\} \subset W^{1,q}_0(\Omega)$, and $u \in W^{1,p}_0(\Omega)$, $v \in W^{1,q}_0(\Omega)$ such that, up to subsequences relabeled the same,
\begin{align*}
	&u_k\rightharpoonup u \ \mbox{weakly} \ \mbox{in} \ W^{1,p}_0(\Omega), \ u_k \to u \ \mbox{in} \ L^{s_1}(\Omega), \ \mbox{and a.e. in} \ \Omega;\\
	\nonumber&v_k\rightharpoonup v \ \mbox{weakly} \ \mbox{in} \ W^{1,q}_0(\Omega), \ v_k \to v \ \mbox{in} \ L^{s_2}(\Omega), \ \mbox{and a.e. in} \ \Omega,
\end{align*}
for $p=\frac{2(r-\theta+1)}{r+\theta+1}$, $q=\frac{2N(1-\theta)}{N-2\theta}$, $u\geqslant 0$, $v\geqslant 0$ a.e. in $\Omega$, and
\begin{equation}
\label{weak}
\|u\|^p_{W^{1,p}_0}+\int_\Omega u^{r-\theta+1} + \int_\Omega u^r v^{1-\theta} \leqslant C\|f\|_{L^m}^{\frac{r-\theta+1}{r+\theta}} \mbox{ and }  \|v\|_{W^{1,q}_0}^q \leq C \|f\|_{L^m}^{\frac{N(r-\theta+1)}{(N-2\theta)(r+\theta)}},
\end{equation}
where $s_1< \max\{p^*,r-\theta+1\}$, $s_2< q^*$, and we obviously took advantage to the fact that $\tau_k \to \infty$ if $k\to \infty$.
At this point, let us stress that $p\geqslant q$ since $r\geqslant \frac{N+2}{N-2}$.

Once more,  we still have to prove  that $g(x,u_k,v_k)\to g(x,u,v)$ and $h(x,u_k,v_k)\to h(x,u,v)$ strongly in $L^1(\Omega)$. For this we employ Lemma \ref{lem48} for $s=p$ and $t=q$. Whereas the proof of the first convergence  is the same, i.e.,  we take $\varphi=\frac{T_\lambda(G_n(u_k))}{\lambda}$ in the first equation of \eqref{aprox}, and after the same computations analogous to the last case,  we get
\begin{align*}
\int_{\{u_k>n\}} u_k^{r-1}v_k^{\theta+1} \leqslant C\int_{\{u_k>n\}} f.
\end{align*}
and then, by \eqref{weak} and Lemma \ref{lem48} we obtain that $
g(x,u_k,v_k) \to g(x,u,v) \ \ \mbox{in} \ L^1(\Omega)$.

Finally, once again by \eqref{weak}  we know that $\{u_k^r\}$ is uniformly integrable. Indeed, note that by Lemma \ref{lem4.10}, since $0<\theta < \min \big (\frac{N+2}{3N-2}, \frac{4}{N-2},\frac{1}{2}\big)$,
\begin{align*}
    \int _E u_k^r &\leqslant \left(\int_E u_k^{r-\theta+1}\right)^\frac{r}{r-\theta+1}\cdot {\rm \mbox{meas}}(E)^{\frac{1-\theta}{r-\theta+1}}
    \\
    &\leqslant  C\bigg(\|f\|_{L^m}^{\frac{r-\theta+1}{r+\theta}}+1\bigg)^\frac{r}{r-\theta+1} {\rm \bigg(\mbox{meas}}(E)\bigg)^{\frac{1-\theta}{r-\theta+1}},
\end{align*}
and then, $\{u_k^r\}$ is uniformly integrable. Furthermore, by \eqref{esti2}, we also have that
\[ \int_{\{v_k>n\}}u_k^r v_k^{1-\theta} \leqslant C\bigg(\|f\|_{L^m}^{\frac{r-\theta+1}{r+\theta}}+1\bigg)\]
so that, once again by Lemma \ref{lem48}
$    h(u,u_k,v_k) \to h(x,u,v) \  \ \mbox{in} \ \ L^1(\Omega)$.

Therefore, by passing to the limit as $k\to+\infty$ in the first and second equations of \eqref{aprox},  we obtain that $(u,v)$ is a solution of \eqref{P}.

\section{Regularizing Effects}
 In this section there are presented the proofs of  Corollaries \ref{C2.4} and \ref{C2.6},
  where we inspect the conditions on the data determining  the presence or absence of regularizing effects of the solutions of \eqref{P}, i.e., in the light of Definition \ref{def2.1}, we show whether the solutions are Lebesgue or Sobolev regularized.
  Despite being direct, for the convenience of the reader we give the details .
\subsection{Proof of Corollary \ref{C2.4}}
\begin{itemize}
 \item [(A)]
    It is easy to see that
    $
         r+\theta +1 > 2^*\iff (r+\theta + 1)^{\prime}<(2^*)^{\prime}.
    $

    Thus, if $(r+\theta +1)^{\prime}<m< (2^*)^{\prime}$,  by Theorem \ref{lemah} we have $u \in W^{1,2}_0(\Omega) \cap L^{r+\theta+1}(\Omega)$. Moreover, since $r+\theta +1 > 2^*$ we obtain the following continuous immersion $L^{r+\theta+1}(\Omega) \subset L^{2^*}(\Omega)$. As $m < (2^*)^{\prime} \iff 2^* > m^{**}$, implies that  $L^{2^*}(\Omega)\subset L^{m^{**}}(\Omega)$ is the continuous immersion.
    Therefore, we have also a regularizing effect for the Lebesgue summability of the solution $u$.

 \item [(B)] Knowing that $r+\theta+1> 2^* \iff (r+\theta+1)^{\prime}<(2^*)^{\prime}$ and $m\geqslant (2^*)^{\prime}$, there follows  $m>(r+\theta+1)^{\prime}$, then by Theorem \ref{lemah} we get $u\in W^{1,2}_0(\Omega)\cap L^{r+\theta+1}(\Omega)$. Moreover, it is easy to see that $
    r+\theta+1 > m^{**}\iff m < \frac{N(r+\theta+1)}{N+2(r+\theta+1)}.
$
Therefore, we have a regularizing effect for the Lebesgue summability of the solution $u$.
\item [(C)]
By Theorem \ref{lemah} we know that $u \in L^{r+\theta+1}(\Omega)$ and $v\in L^{2^*}(\Omega)$, so it is easy to see that $u^r \in L^{\frac{r+\theta+1}{r}}(\Omega)$ and $|v|^\theta \in L^{\frac{2^*}{\theta}}(\Omega)$. Thus, by  interpolation, we get
\begin{align*}
    ||h(x,u,v)||_{L^s}= \Big(\int _{\Omega} |h(x,u,v)|^s\Big)^\frac{1}{s}\leqslant \Big(d^s_1 \int_{\Omega} u^{sr} v^{s\theta}  \Big)^\frac{1}{s} \leqslant d_2||u^r||_{L^{\frac{r+\theta+1}{r}}} || v^{\theta}||_{L^{\frac{2^*}{\theta}}} < \infty,
\end{align*}
where for $s \geqslant 1$
\begin{align*}
    \frac{1}{s} = \frac{r}{r+\theta+1} + \frac{\theta}{2^*} \implies s = \frac{2^*(r+\theta+1)}{2^*r+\theta(r +\theta + 1)}.
\end{align*}
It is enough to prove that $1\leqslant s < (2^*)^{\prime}$, or equivalently,
\[\frac{2^*}{2^* - \theta}\leqslant \frac{r+\theta+1}{r} < \frac{2^*}{2^* -(\theta+1)}\]
and thus $v$ is going to be Sobolev regularized.
For this, remark that, since $\Big(\frac{r+\theta+1}{r}\Big)^{\prime} = \frac{r+\theta+1}{\theta+1}$, $\Big(\frac{2^*}{2^*-\theta}\Big)^{\prime} = \frac{2^*}{\theta}$ and $\Big(\frac{2^*}{2^* - (\theta+1)}\Big)^{\prime}= \frac{2^*}{\theta+1}$, and by hypotheses, we have
\[2^*< r+\theta+1\leqslant\frac{2^*(\theta+1)}{\theta},\]
\end{itemize}
proving that $v$ is Sobolev regularized.
%\textcolor{blue}
\subsection{Proof of Corollary \ref{C2.6}}

\begin{itemize}
%    \item[(A)] First, remark that $u\in W^{1,p}_0(\Omega)$, by the Sobolev embedding we get $u\in L^{p^*}(\Omega)$. Thus, since $r-\theta+1>2^*>2^*(1-\theta)$, there follows
%$    p^*< r-\theta+1 \iff 2^*(1-\theta)< r-\theta+1$. Thus, in order to prove that $u$ is Lebesgue regularized, it is enough to contrast $r-\theta+1$ with $m^{**}$. Indeed,  observe that $m < (2^*)^{\prime} \iff m^{**} < 2^*.$ However, by hypothesis we know that $r-\theta+1 > 2^*$  which implies that $r-\theta+1> m^{**}$, that is,  $L^{r-\theta+1}(\Omega)\subset L^{m^{**}}(\Omega),$
 %strictly, and thence, $u$ is Lebesgue regularized.

\item[(A)] In order to show that $u$ is Sobolev regularized, remark that it is enough prove that $p>m^*$. Indeed, it is clear that
$     p>m^* \iff m < \frac{2N(r-\theta+1)}{N(r+\theta+1)+2(r-\theta+1)} $.
Thus, since $u\in W^{1,p}(\Omega)$ we obtain that $u$ is Sobolev regularized. Moreover, as it turns out, once again to prove that $u$ is Lebesgue regularized we shall concentrate in $L^{r-\theta+1}(\Omega)$. Indeed,  $p^*< r-\theta+1$, since $2^*(1-\theta)< r-\theta+1$. In this fashion, let us prove that under our hypotheses, $r-\theta+1>m^{**}$.
It is easy to see that $ r-\theta + 1 > 2^*(1-\theta) $ is equivalent to $
    \frac{N(r-\theta + 1)}{N+2(r-\theta+1)} > \frac{2N(r-\theta+1)}{N(r+\theta+1)+2(r-\theta+1)}$. Then, since
$     \frac{2N(r-\theta+1)}{N(r+\theta+1)+2(r-\theta+1)} > m$, we have
$  \frac{N(r-\theta+1)}{N+2(r-\theta+1)}> m$, what implies that $r-\theta+1> m^{**}$, and consequently, $u$ is Lebesgue regularized.

\item[(B)] Notice that $q^*=2^*(1-\theta)$, where $ q=\frac{2N(1-\theta)}{N-2\theta}$. Moreover, for  $p=\frac{2(r-\theta+1)}{r+\theta+1}$, by Theorem \ref{principal} we know that $u \in L^{r-\theta+1}(\Omega) \cap W^{1,p}_{0}(\Omega)$ and $v\in L^{q^*}(\Omega)$, or equivalently, $u^r \in L^{\frac{r-\theta+1}{r}}(\Omega)$ and $v^\theta \in L^{\frac{q^*}{\theta}}(\Omega)$. Now, by combining \eqref{P3} with the standard interpolation inequality for Lebesgue Spaces, it is obvious that
%\begin{align*}
%    |g(x,u,v)| \leqslant d_2|u|^r|v|^\theta %\ \mbox{a.e in } \ \Omega.
%\end{align*}
%Thus, by interpolation inequality
\[ h(x,u,v) \in L^s(\Omega)  \mbox{ where }s= \frac{2^*(r-\theta+1)(1-\theta)}{2^*r(1-\theta)+\theta(r-\theta+1)},\]
and $s\geqslant 1$ as a consequence of $r-\theta+1\geqslant \dfrac{2^*(1-\theta)^2}{\theta}.$

In order to show that $v$ is Sobolev regularized, remark that  $q> s^*$ is  equivalently to $ r-\theta+1> 2^*(1-\theta).$ Indeed, first remark that
\begin{equation}
\label{1550}
s=\dfrac{2N(1-\theta)(r-\theta+1)}{2N(1-\theta)r+(N-2)\theta(r-\theta+1)}
\end{equation}
and $q>s^*$ is equivalent to
\begin{equation}
\label{1553}
\dfrac{Nq}{N+q}>s \iff \dfrac{2N(1-\theta)}{N+2(1-2\theta)}>s.
\end{equation}
Then, by \eqref{1550} and \eqref{1553}, after certain basic computations, we have that $q>s^*$ is equivalent to
\begin{align*}
2N(1-\theta)r+(N-2)\theta (r-\theta+1) & > (N+2(1-2\theta))(r-\theta+1)\\
\iff&\\
2N(1-\theta)r > (r-\theta+1)&(N(1-\theta)+2(1-\theta))
\\
\iff&\\
(2^*)^\prime = \dfrac{2N}{N+2} &> \dfrac{r-\theta+1}{r}.
\end{align*}
However, since \[\bigg( \dfrac{r-\theta+1}{r} \bigg)^\prime = \dfrac{r-\theta+1}{1-\theta}\]
by the later equivalencies, we arrive at
\[q>s^* \iff \dfrac{r-\theta+1}{1-\theta}>2^*,\]
as we have claimed. In this way, since $v\in W^{1,q}_0(\Omega)$ it is clear that $v$ is Sobolev regularized.
\end{itemize}

%\section*{Acknowlegments}

%The first author was supported by CAPES - Brazil grant 88887.480897/2020-00 and  CNPq - Brazil grant 142279/2020-0. The second author was partially supported by CAPES - Print/Brazil 88887.469671/2019-00, FAPDF/Brazil grant 00193.00002176/2018-87, and FEMAT/Brazil grant 01/2023. The authors would like to express their gratitude to the anonymous referee for her/his careful reading and her/his constructive corrections and suggestions to the text. Thanks to her/his work, we were able to improve the present paper.

\bibliographystyle{plain}

\end{document}